\documentclass[11pt]{article}

\usepackage{geometry}
\usepackage{amsthm}
\usepackage{amsmath}
\usepackage{amssymb}
\usepackage{mathtools}
\usepackage{mathrsfs}
\usepackage{subfigure}
\usepackage{multicol}
\usepackage{tikz}

\newcommand{\bb}{\mathbb}
\newcommand{\mc}{\mathcal}
\newcommand{\ms}{\mathscr}
\newcommand{\abs}[1]{\lvert #1 \rvert}

\newcommand{\minus}{\,\backslash\,}
\newcommand{\contract}{\, / \,}

\DeclareMathOperator{\link}{link}

\DeclareMathOperator{\conv}{conv}

\DeclareMathOperator{\sh}{shape}
\DeclareMathOperator{\spn}{span}

\theoremstyle{definition}

\newtheorem{thm}{Theorem}[section]
\newtheorem{prop}[thm]{Proposition}
\newtheorem{lem}[thm]{Lemma}

\newtheorem{claim}[thm]{Claim}

\newtheorem{defn}[thm]{Definition}

\numberwithin{equation}{section}

\title{Flip-connectivity of triangulations of the product of a tetrahedron and simplex}
\author{Gaku Liu}

\begin{document}
\maketitle

\begin{abstract}
A flip is a minimal move between two triangulations of a polytope. An open question is whether any two triangulations of the product of two simplices can be connected through a series of flips. This was proven in the case where one of the simplices is a triangle by Santos in 2005. In this paper we extend this to when one of the simplices is a tetrahedron.
\end{abstract}

\section{Introduction}

The triangulations of the product of two simplices are a well-studied and interesting set of objects. They have connections to algebraic geometry and commutative algebra and are an important tool to understanding triangulations of products of other polytopes. See \cite[Chapter 6.2]{DRS} for an overview. They have also been extensively studied for their own sake and have been characterized as tropical hyperplane arrangements, fine mixed subdivisions, and tropical oriented matroids \cite{DS,San05,AD}.

A major open question is whether or not the set of triangulations of the product of two simplices is \emph{flip-connected}.
A flip can be thought of as a minimal move between two triangulations of a fixed polytope. The set of triangulations of a polytope is flip-connected if any two triangulations can be connected by a series of flips. It has been long known that every two-dimensional polygon has a flip-connected set of triangulations. On the other hand, there are examples in higher dimensions of polytopes with non-flip-connected triangulations \cite{San00}. However, there is very little known about flip-connectivity even in specific cases or in low dimension; for example, it is unknown whether all three-dimensional polytopes have a flip-connected set of triangulations. In the case of the product of two simplices, Santos \cite{San05} proved that the triangulations of $\Delta^2 \times \Delta^n$ are flip-connected for any $n$. In this paper, we extend this result to triangulations of $\Delta^3 \times \Delta^n$.

In broad terms, our proof gives an algorithm for applying flips to a triangulation until it reaches a specific fixed triangulation. Each step of the algorithm consists of two parts:
\begin{enumerate}
\item Locate a special circuit of $\Delta^3 \times \Delta^{n-1}$, and consider the set $\ms S$ of all simplices of the triangulation which contain the negative elements of this circuit.
\item Apply a series of flips such that the only simplices affected are in $\ms S$, until eventually the circuit itself can be flipped.
\end{enumerate}
While the first part takes up a relatively small portion of this paper, it is the key idea of the proof. We will find our special circuit using a quasiorder defined on all the simplices of the triangulation.

The paper is organized as follows. Section 2 gives an overview of triangulations and flips in general and defines some other important concepts, in particular contraction. Section 3 discusses triangulations of the product of two simplices. Section 4 develops the machinery we will use in the main proof. Section 5 is the main proof.

\section{Triangulations of general sets}

\subsection{Triangulations and flips}

We will give a brief overview of triangulations and flips following \cite{DRS}.
Throughout this section, we let $A$ be a finite set of points in $\bb R^d$, not necessarily in convex position. We will use \emph{simplex} to mean an affinely independent set of points. In other words, a simplex will refer to the set of vertices of a geometric simplex, rather than the whole polytope itself. Let $\conv(S)$ denote the convex hull of a point set $S$.

\begin{defn}
A \emph{triangulation} of $A$ is a collection $\ms T$ of simplices which are subsets of $A$ such that
\begin{enumerate}
\item Any subset of a simplex in $\ms T$ is in $\ms T$.
\item Any two simplices $\sigma_1$, $\sigma_2 \in \ms T$ intersect properly; that is, $\conv(\sigma_1) \cap \conv(\sigma_2) = \conv(\sigma_1 \cap \sigma_2)$.
\item $\bigcup_{\sigma \in \ms T} \conv(\sigma) = \conv(A)$.
\end{enumerate}
\end{defn}

Given a triangulation $\ms T$, let $\ms T^\ast$ denote the subcollection of maximal simplices of $\ms T$. Note that $\ms T$ is determined by $\ms T^\ast$.

We need to give a few more definitions before introducing flips. A \emph{circuit} is a minimal affinely dependent subset of $\bb R^d$. If $X = \{x_1,\dotsc,x_k\}$ is a circuit, then the points of $X$ satisfy an affine dependence equation
\[
\sum_{i=1}^k \lambda_i x_i = 0
\]
where $\lambda_i \in \bb R \minus \{0\}$ for all $i$, $\sum_i \lambda_i = 0$, and the equation is unique up to multiplication by a constant. This gives a unique partition $X = X^+ \cup X^-$ of $X$ given by $X^+ = \{x_i : \lambda_i > 0\}$ and $X^- = \{x_i : \lambda_i < 0\}$. We will write this as $X = \{X^+, X^-\}$.

For any circuit $X = \{X^+,X^-\}$, the relative interiors of $\conv(X^+)$ and $\conv(X^-)$ intersect. In particular, $X^+$ and $X^-$ cannot both be elements of the same triangulation. The circuit $X$ has exactly two triangulations, given by
\[
\ms T_X^+ := \{ \sigma \subseteq X : \sigma \not\supseteq X^+ \} \qquad \ms T_X^- := \{ \sigma \subseteq X : \sigma \not\supseteq X^- \}.
\]

Given a triangulation $\ms T$ and a simplex $\sigma \in \ms T$, we define the \emph{link} of $\sigma$ in $\ms T$ as
\[
\link_{\ms T}(\sigma) := \{ \rho \in \ms T : \rho \cap \sigma = \emptyset,\, \rho \cup \sigma \in \ms T \}.
\]
We are now ready to state the definition of a flip, in the form of a proposition.

\begin{prop}[\cite{San02}] \label{flipdef}
Let $\ms T$ be a triangulation of $A$. Suppose there is a circuit $X = \{X^+,X^-\}$ contained in $A$ such that
\begin{enumerate}
\item $\ms T_X^+ \subseteq \ms T$.
\item All maximal simplices of $\ms T_X^+$ have the same link $\ms L$ in $\ms T$.
\end{enumerate}
Then the collection
\[
\ms T' := \ms T \minus \{ \rho \cup \sigma : \rho \in \ms L, \sigma \in \ms T_X^+ \} \cup \{ \rho \cup \sigma : \rho \in \ms L, \sigma \in \ms T_X^- \}
\]
is a triangulation of $A$. We say that $\ms T$ has a \emph{flip} supported on $(X^+,X^-)$, and that $\ms T'$ is the result of applying this flip to $\ms T$.
\end{prop}

If two triangulations of $A$ can be connected by a series of flips, then we say that these triangulations are \emph{flip-connected}. If every pair of triangulations of $A$ is flip-connected, then the set of triangulations of $A$ is \emph{flip-connected}.


\subsection{Vector configurations}

In order to talk about contractions, we will need to consider more general configurations than point sets. A \emph{vector configuration} is a finite collection of labeled nonzero vectors in $\bb R^d$. A vector configuration may contain multiple instances of the same vector with different labels. (Formally, a vector configuration is a map $A : I \to \bb R^d \minus \{0\}$ where $I$ is a finite label set.) Given a set $S$ of vectors, let $\spn(S)$ denote the nonnegative span of $S$, where $\spn(\emptyset)$ is defined to be $\{0\}$.

\begin{defn}
A \emph{triangulation} of a vector configuration $A$ is a collection $\ms T$ of linearly independent subsets of $A$ such that
\begin{enumerate}
\item Any subset of an element of $\ms T$ is in $\ms T$.
\item For any two elements $\sigma_1$, $\sigma_2 \in \ms T$ we have $\spn(\sigma_1) \cap \spn(\sigma_2) = \spn(\sigma_1 \cap \sigma_2)$.\footnote{In the intersection $\sigma_1 \cap \sigma_2$, we consider two vectors with different labels to be different. In particular, this condition implies that if $u$ and $v$ are the same vector with different labels, then $\{u\}$ and $\{v\}$ cannot both be in $\ms T$.}
\item $\bigcup_{\sigma \in \ms T} \spn(\sigma) = \spn(A)$.
\end{enumerate}
\end{defn}

We call a linearly independent set of vectors a \emph{simplex}. We will make it clear through context when we should consider a set $A$ a set of points or a vector configuration. Given a point set $A \subset \bb R^d$, we can convert $A$ into a vector configuration by embedding $A$ into a hyperplane $H \subset \bb R^{d+1}$ which does not contain the origin. The triangulations of $A$ and this corresponding vector configuration are the same.

If $A$ is a vector configuration that does not contain different labeled vectors $u$, $v$ which are parallel to each other (i.e., positive multiples of each other), then we call $A$ \emph{simple}. For any vector configuration $A$, we define the \emph{simplification} of $A$ to be a minimal vector configuration $\hat A$ such that any vector of $A$ is parallel to some element of $\hat A$. Given a triangulation of $A$, we obtain a unique triangulation of $\hat A$ by replacing each element of $A$ with its parallel vector in $\hat A$.

\subsection{Restriction and contraction}

We now define two operations on triangulations. Let $\ms T$ be a triangulation of a point set $A$. Let $F$ be the intersection of some face of $\conv(A)$ with $A$. Then $\ms T$ induces a triangulation of $F$ given by
\[
\ms T[F] := \{ \sigma \in \ms T : \sigma \subseteq F \}.
\]
We call this the \emph{restriction} of $\ms T$ to $F$.

Now, embed $A$ as vector configuration in $\bb R^d$. Let $S$ be a nonempty intersection of some linear subspace of $\bb R^d$ with $A$. Let $\pi : \bb R^d \to \bb R^{d'}$ be a linear map whose kernel is the linear span of $S$. We define the \emph{contraction} of $A$ at $S$ to be the vector configuration
\[
A \contract S := \{ \pi(y) : y \in A \minus S\}
\]
where the right hand side is a multiset with label set $A \minus S$.

Suppose $\xi \subseteq S$ is a simplex with the same dimension as $\conv(S)$. Then we have the following.

\begin{prop} \label{contractionfacts}
The following are true.
\begin{enumerate}
\item The map $\sigma \mapsto \pi(\sigma \minus \xi)$ gives a bijection from the set of simplices of $A$ containing $\xi$ to the set of simplices of $A \contract S$.
\item The above map preserves codimensions of simplices.
\item If $\sigma_1$, $\sigma_2 \subseteq A$ are simplices which contain $\xi$ and intersect properly, then $\pi(\sigma_1 \minus \xi)$ and $\pi(\sigma_2 \minus \xi)$ intersect properly and their intersection is $\pi((\sigma_1 \cap \sigma_2) \minus \xi)$.
\end{enumerate}
\end{prop}

\begin{prop}
Let $\ms T$ be a triangulation of $A$ which contains $\xi$. Then the collection
\[
\ms T \contract \xi := \{ \pi(\sigma \minus \xi) : \xi \subseteq \sigma,\, \sigma \in \ms T \}
\]
is a triangulation of $A \contract S$.
\end{prop}

\subsection{Local triangulations}

In this section we give an easy generalization of triangulations which we will need for technical reasons. Let $A \subseteq \bb R^d$ be a point configuration, and let $\xi$ be a simplex of $A$.

\begin{defn}
For $\xi$ nonempty, a \emph{local triangulation} of $A$ at $\xi$ is a collection $\ms T$ of simplices contained in $A$ such that
\begin{enumerate}
\item All simplices in $\ms T$ contain $\xi$.
\item If $\sigma \subseteq \ms T$ and $\xi \subseteq \sigma' \subseteq \sigma$, then $\sigma' \subseteq \ms T$.
\item Any two simplices in $\ms T$ intersect properly.
\item There is a an open set $U \subset \bb R^d$ with $U \cap \xi \neq \emptyset$ such that
\[
\left( \bigcup_{\sigma \in \ms T} \conv(\sigma) \right) \cap U = \conv(A) \cap U.
\]
\end{enumerate}
If $\xi = \emptyset$, then we define a local triangulation of $A$ at $\xi$ to be a triangulation of $A$.
\end{defn}

Let $\ms T^\ast$ denote the set of maximal simplices of a local triangulation $\ms T$. If $\ms T$ is a local triangulation of $A$ at (possibly empty) $\zeta$ and $\xi \subseteq A$ is a simplex, let $\ms T(\xi)$ be the subcollection of simplices of $\ms T$ which contain $\xi$. If $\xi \cup \zeta \in \ms T$, then $\ms T(\xi)$ is a local triangulation of $A$ at $\xi \cup \zeta$.

Suppose $\xi \subseteq A$ is a nonempty simplex. Embed $A \subset \bb R^d$ as a vector configuration, and let $S$ be the intersection of $A$ with the linear span of $\xi$. Let $\pi : \bb R^d \to \bb R^{d'}$ be a linear map whose kernel is this span. Then any local triangulation $\ms T$ at $\xi$ gives a triangulation
\[
\ms T \contract \xi := \{ \pi(\sigma \minus \xi) : \sigma \in \ms T \}
\]
of $A \contract S$. In fact, this construction is a bijection from the set of local triangulations of $A$ at $\xi$ to the set of triangulations of $A \contract S$.

\section{The product of two simplices}

\subsection{The oriented matroid of $\Delta^{m-1} \times \Delta^{n-1}$}

We now take $A = \Delta^{m-1} \times \Delta^{n-1}$, the product of two simplices of dimensions $m-1$ and $n-1$. Following the conventions of the previous section, we will understand $\Delta^{m-1} \times \Delta^{n-1}$ to mean the set of vertices of $\Delta^{m-1} \times \Delta^{n-1}$ rather than the polytope itself.

The vertices of $\Delta^{m-1} \times \Delta^{n-1}$ are given by
\[
\{e_i \times f_j : i \in [m], j \in [n] \}
\]
where $[l] := \{1,2,\dotsc,l\}$ and $\Delta^{m-1} := \{e_1,\dotsc,e_m\}$, $\Delta^{n-1} := \{f_1,\dotsc,f_n\}$. Consider a bipartite directed graph $G_{m,n}$ with vertex set $\Delta^{m-1} \cup \Delta^{n-1}$ and directed edges $e_if_j$ for each $i$, $j$. We have bijection $e_i \times f_j \mapsto e_if_j$ from $\Delta^{m-1} \times \Delta^{n-1}$ to edges of $G_{m,n}$. In this map, the circuits of $\Delta^{m-1} \times \Delta^{n-1}$ map bijectively to undirected cycles of $G_{m,n}$. If we traverse a cycle $C$ of $G_{m,n}$, then we travel along some edges in the forward direction and the other edges in the backward direction; this gives a partition $C^+ \cup C^-$ of the edges of $C$. This partition is the same as the partition $X = \{X^+,X^-\}$ of the circuit $X \subseteq \Delta^{m-1} \times \Delta^{n-1}$ corresponding to $C$.

All of this can be summarized by saying that the map $e_i \times f_j \mapsto e_if_j$ gives an isomorphism of oriented matroids from the
oriented matroid of affine dependencies of $\Delta^{m-1} \times \Delta^{n-1}$ to the oriented matroid of $G_{m,n}$.

Through the above map, each simplex $\sigma \subseteq \Delta^{m-1} \times \Delta^{n-1}$ maps to a set of edges $E(\sigma)$ of $G_{m,n}$ with no cycles. Let $T(\sigma)$ denote the graph with vertex set $\Delta^{m-1} \cup \Delta^{n-1}$ and edge set $E(\sigma)$. If $\sigma$ is maximal, then $T(\sigma)$ is a spanning tree of $G_{m,n}$.

Let $\sigma \subseteq \Delta^{m-1} \times \Delta^{n-1}$ be a simplex. For each vertex $v$ of $G_{m,n}$, define $N_\sigma(v)$ to be the undirected neighborhood of $v$ in $T(\sigma)$. Suppose that $j_1$, \dots, $j_k$ are all the indices $j \in [n]$ for which $\abs{N_\sigma(f_j)} > 1$. We define the \emph{shape} of $\sigma$ to be the set
\[
\sh(\sigma) := \{ N_\sigma(f_{j_1}), \dotsc, N_\sigma(f_{j_k}) \}.
\]

Finally, given a triangulation (or local triangulation) $\ms T$ of $\Delta^{m-1} \times \Delta^{n-1}$, we define $\ms T^\star$ to be the subcollection of simplices $\sigma \in \ms T$ such that $N_\sigma(f_j) \neq \emptyset$ for all $j \in [n]$. In particular, $\ms T^\ast \subseteq \ms T^\star$.

\subsection{The Cayley trick}

The Cayley trick provides a useful way to visualize triangulations of $\Delta^{m-1} \times \Delta^{n-1}$ by mapping them to lower-dimensional objects called fine mixed subdivisions. Let $\sigma$ be a simplex in $\Delta^{m-1} \times \Delta^{n-1}$ with $N_\sigma(f_j) \neq \emptyset$ for all $j \in [n]$. Let $\mc C(\sigma)$ denote the polytope
\[
\mc C(\sigma) := \sum_{j=1}^n \conv(N_\sigma(f_j)) = \{ x_1 + \dotsb + x_n : x_j \in \conv(N_\sigma(f_j)) \}
\]
where we view each $N_\sigma(f_j)$ as a subset of $\Delta^{m-1}$.
We call $\mc C(\sigma)$ the \emph{fine mixed cell} associated to $\sigma$.
If $\ms T$ is a triangulation, then the collection $\mc C(\ms T) := \{ \mc C(\sigma) \}_{\sigma \in \ms T^\star}$ gives a subdivision of the polytope $n\Delta^{m-1} = \{nx : x \in \conv(\Delta^{m-1})\}$ which we call a \emph{fine mixed subdivision} of $n\Delta^{m-1}$. Moreover, the map $\sigma \mapsto \mc C(\sigma)$ gives a bijection $\ms T^\star \to \mc C(\ms T)$ which preserves face relations (and therefore adjacency relations) when viewed as a map of polytopes.

The \emph{shape} of $\mc C(\sigma)$ is defined to be $\sh(\sigma)$. If $\mc C(\sigma)$ has shape $\{N_1,\dotsc,N_k\}$, then it is geometrically congruent to $\sum_{r=1}^k \conv(N_r)$. We call a fine mixed cell with shape $\{\Delta^{m-1}\}$ an \emph{unmixed cell}. If $\mc C(\sigma)$ is an unmixed cell, then $\sigma$ is maximal, and there is a unique $f_j \in \Delta^{n-1}$ such that $N_{\sigma}(f_j) = \Delta^{m-1}$. In this case, we ``label'' the unmixed cell $\mc C(\sigma)$ with $f_j$. If $\ms T$ is a triangulation, this labeling gives a bijection between $\Delta^{n-1}$ and the set of unmixed cells of $\mc C(\ms T)$. A fact we will not use is that a triangulation $\ms T$ of $\Delta^{m-1} \times \Delta^{n-1}$ is completely determined by $\mc C(\ms T)$ and the labeling of the umixed cells of $\mc C(\ms T)$.

Now suppose $\ms T$ is a local triangulation of $\Delta^{m-1} \times \Delta^{n-1}$ at $e_1 \times f_1$. The simplices $\{e_1\} \times \Delta^{n-1}$ and $\Delta^{m-1} \times \{f_1\}$ are faces of $\Delta^{m-1} \times \Delta^{n-1}$ that contain $e_1 \times f_1$, and therefore must be elements of $\ms T$. Thus, the collection $\mc C(\ms T)$ is a connected set of fine mixed cells which contains the cell $\{ne_1\}$ and some unmixed cell with label $f_1$.

\subsection{Triangulations of $\Delta^1 \times \Delta^{n-1}$} \label{sec:m=1}
We will use the Cayley trick to analyze triangulations and local triangulations of $\Delta^1 \times \Delta^{n-1}$. This case will be an important starting point when working with higher values of $m$.

Let $\ms T$ be a triangulation of $\Delta^1 \times \Delta^{n-1}$. Every maximal fine mixed cell of $\mc C(\ms T)$ is a unit line segment, and they are arranged in a line along $n\Delta^1$. The maximal cell $\mc C(\tau)$ incident to $ne_1$ must have $N_\tau(e_1) = \Delta^{n-1}$.

Consider two adjacent maximal cells $\mc C(\tau)$ and $\mc C(\tau')$. Let $\mc C(\sigma)$ be the shared vertex of $\mc C(\tau)$ and $\mc C(\tau')$. Then $\tau$ and $\tau'$ are maximal simplices of $\ms T$ that share the facet $\sigma$. Suppose $ne_1$ is closer to $\mc C(\tau)$ than to $\mc C(\tau')$. It follows that $\tau = \sigma \cup \{e_1 \times f_j\}$ for some $f_j \in N_\sigma(e_2)$, and $\tau' = \sigma \cup \{e_2 \times f_{j'}\}$ for some $f_{j'} \in N_\sigma(e_1)$. Thus,
\begin{equation} \label{adjacent}
\tau' = \tau \minus \{e_1 \times f_j\} \cup \{e_2 \times f_{j'}\}.
\end{equation}
The $f_j$ and $f_{j'}$ are precisely the ``labels'' of $\mc C(\tau)$ and $\mc C(\tau')$ respectively; that is, $f_j$ is such that $N_\tau(f_j) = \Delta^1$.

Suppose instead that $\ms T$ is a local triangulation at $e_1 \times f_1$. Then the set of maximal cells of $\mc C(\ms T)$ is a connected set of line segments, one of which is incident to $ne_1$ and one of which has label $f_1$. We can apply the same reasoning to adjacent cells of $\mc C(\ms T)$ as above. In particular, the cell $\mc C(\tau)$ with label $f_1$ must be the furthest cell from $ne_1$: If there were a cell $\mc C(\tau')$ adjacent to $\mc C(\tau)$ and further away from $ne_1$, then by \eqref{adjacent}, $\tau'$ would not contain $e_1 \times f_1$, contradicting the definition of $\ms T$.

Finally, suppose $\ms T$ is a local triangulation at $\{e_1 \times f_1, e_2 \times f_2\}$. It is not hard to see that $\ms T$ must contain the simplices
\[
\{e_1 \times f_1, e_2 \times f_2, e_1 \times f_2\}, \{e_1 \times f_1, e_2 \times f_2, e_2 \times f_1\}.
\]
Thus, the set of maximal cells of $\mc C(\tau)$ is a connected set of line segments, one of which has label $f_2$ and another of which has label $f_1$. We can use the above reasoning to show that the cell labeled $f_2$ is the cell closest to $ne_1$ and the cell labeled $f_1$ is the furthest.

We can summarize this as follows.

\begin{prop} \label{m=1}
Let $\ms T$ be a local triangulation of $\Delta^1 \times \Delta^{n-1}$ at $\xi \subseteq \{e_1 \times f_1, e_2 \times f_2\}$. Then there exists a unique ordering $\tau_1$, \dots, $\tau_N$ of $\ms T^\ast$ such that
\begin{enumerate}
\item For each $r = 1$, \dots, $N-1$, $\tau_r$ is adjacent to $\tau_{r+1}$, and these are the only pairs of adjacent maximal simplices.
\item We have
\[
\tau_{r+1} = \tau_r \minus \{e_1 \times f_{j_r}\} \cup \{e_2 \times f_{j_{r+1}}\},
\]
where $f_{j_r}$ is such that $N_{\tau_{r}}(f_{j_r}) = \Delta^1$.
\end{enumerate}
Furthermore,
\begin{enumerate} \setcounter{enumi}{2}
\item If $e_1 \times f_1 \in \xi$, then $f_{j_N} = f_1$; otherwise, $N_{\tau_N}(e_2) = \Delta^{n-1}$.
\item If $e_2 \times f_2 \in \xi$, then $f_{j_1} = f_2$; otherwise, $N_{\tau_1}(e_1) = \Delta^{n-1}$.
\end{enumerate}
\end{prop}

Let $\tau_1$, \dots, $\tau_N$ be the ordering of the maximal simplices of $\ms T$ given by Proposition~\ref{m=1}. We formalize this order as $<_{\ms T}$, so that
\begin{equation} \label{simplexorder}
\tau_1 <_{\ms T} \tau_2 <_{\ms T} < \dotsb <_{\ms T} \tau_N.
\end{equation}
Let $f_{j_1}$, \dots, $f_{j_N}$ be as in the Proposition. We have a total order $<_{\ms T}$ on the $\{f_{j_r}\}$ given by
\[
f_{j_1} <_{\ms T} f_{j_2} <_{\ms T} \dotsb <_{\ms T} f_{j_N}.
\]
We can interpret these orders as follows: If $j \in \{j_1,\dotsc,j_N\}$, then any maximal simplex with a label greater than or equal to $f_j$ contains $e_2 \times f_j$, and any maximal simplex with a label less than or equal to $f_j$ contains $e_1 \times f_j$.

We can use this idea to extend $<_{\ms T}$ to a quasiorder on $\Delta^{n-1}$ as follows. Let $f_{j_1'}$, \dots, $f_{j_{N'}'}$ be the elements of $N_{\tau_1}(e_2) \minus \{f_{j_1}\}$, and let $f_{j_1''}$, \dots, $f_{j_{N''}''}$ be the elements of $N_{\tau_N}(e_1) \minus \{f_{j_N}\}$. We define the quasiorder $\le_{\ms T}$ on $\Delta^{n-1}$ by
\begin{equation} \label{vertexorder}
f_{j_1'} =_{\ms T} \dotsb =_{\ms T} f_{j_{N'}'} <_{\ms T} f_{j_1} <_{\ms T}  \dotsb <_{\ms T} f_{j_N} <_{\ms T} f_{j_1''} =_{\ms T} \dotsb =_{\ms T} f_{j_{N''}''}.
\end{equation}
If $\ms T$ is a triangulation, then $\le_{\ms T}$ is a total order on $\Delta^{n-1}$.

We have the following easy way to determine the relative order of two elements.

\begin{prop} \label{compare12}
Let $f_j$, $f_{j'} \in \Delta^{m-1}$ be distinct. Then $f_j <_{\ms T} f_{j'}$ if and only if
\[
\{e_2 \times f_{j}, e_1 \times f_{j'}\} \subseteq \sigma \in \ms T.
\]
for some $\sigma$.
\end{prop}

\begin{proof}
Suppose $f_j <_{\ms T} f_{j'}$. If there is a maximal simplex of $\ms T$ with label $f_j$, let $\tau$ be that simplex. Otherwise, since $f_j <_{\ms T} f_{j'}$, we must have $f_j \in N_{\tau_1}(e_2) \minus \{f_{j_1}\}$; in this case, let $\tau = \tau_1$. Similarly, let $\tau'$ be the maximal simplex of $\ms T$ with label $f_{j'}$ if such a simplex exists, and let $\tau' = \tau_N$ otherwise. In any case, we have $e_2 \times f_j \in \tau$, $e_1 \times f_{j'} \in \tau'$, and $\tau \le_{\ms T} \tau'$. By Proposition~\ref{m=1}, we have $N_\tau(e_1) \supseteq N_{\tau'}(e_1)$. Hence $\{e_2 \times f_{j}, e_1 \times f_{j'}\} \subseteq \tau \in \ms T$, as desired.

Conversely, suppose that $\{e_2 \times f_{j}, e_1 \times f_{j'}\} \subseteq \sigma \in \ms T$. Let $\tau_0$ be a maximal simplex of $\ms T$ containing $\sigma$. Let $\tau$ be the smallest simplex in the order \eqref{simplexorder} containing $e_2 \times f_j$, and let $\tau'$ be the largest simplex in this order containing $e_1 \times f_{j'}$. By Proposition~\ref{m=1}, either $\tau$ has label $f_j$ or $f_j \in N_{\tau_1}(e_2) \minus \{f_{j_1}\}$. Similarly, either $\tau'$ has label $f_{j'}$ or $f_{j'} \in N_{\tau_N}(e_1) \minus \{f_{j_N}\}$. In any case, since $\tau \le_{\ms T} \tau_0 \le_{\ms T} \tau'$ and $f_j$ and $f_{j'}$ are distinct, we have $f_j <_{\ms T} f_{j'}$.
\end{proof}

\subsection{Restriction and contraction} \label{sec:rescon}

In this section we look at restrictions and contractions of triangulations of $\Delta^{m-1} \times \Delta^{n-1}$.

The faces of $\Delta^{m-1} \times \Delta^{n-1}$ are the sets $I \times J$ where $I \subseteq \Delta^{m-1}$ and $J \subseteq \Delta^{n-1}$. Given a triangulation $\ms T$ of $\Delta^{m-1} \times \Delta^{n-1}$, the restriction of $\ms T$ to $I \times J$ is
\[
\ms T[I \times J] = \{ \sigma \in \ms T : \sigma \subseteq I \times J \}.
\]

We now consider contraction. Let $\xi \subseteq \Delta^{m-1} \times \Delta^{n-1}$ be a simplex. Let $I = (I_1,\dotsc,I_s)$ be a partition of $\Delta^{m-1}$ such that each $I_r$ is the intersection of $\Delta^{m-1}$ with a connected component of $T(\xi)$. Let $J = (J_1, \dotsc, J_t)$ be a partition of $\Delta^{n-1}$ such that each $J_r$ is the intersection of $\Delta^{n-1}$ with a connected component of $T(\xi)$. We say that these partitions are \emph{associated} to $\xi$.

Let $C_1$, \dots, $C_k$ be the connected components of $T(\xi)$ containing an edge, and let $p_1$, \dots, $p_k$ and $q_1$, \dots, $q_k$ be the indices such that for each $r = 1$, \dots, $k$, $I_{p_r} \cup J_{q_r}$ is the vertex set of $C_r$.
Then the intersection of $\Delta^{m-1} \times \Delta^{n-1}$ with the affine span of $\xi$ is
\[
S := (I_{p_1} \times J_{q_1}) \cup \dotsb \cup (I_{p_k} \times J_{q_k}).
\]

We now compute the contraction $\Delta^{m-1} \times \Delta^{n-1} \contract S$. Let $V$ be a real vector space with basis $\Delta^{m-1} \cup \Delta^{n-1}$, and embed $\Delta^{m-1} \times \Delta^{n-1} \subset V$ as $e_i \times f_j \mapsto e_i - f_j$. Let $V'$ be a real vector space which has the set of connected components of $T(\xi)$ as a basis. (That is, $V'$ consists of formal linear combinations of the connected components of $T(\xi)$.) Let $\pi : V \to V'$ be the linear map which takes $v \in \Delta^{m-1} \cup \Delta^{n-1}$ to the connected component of $T(\xi)$ that contains $v$. Then the kernel of $\pi$ is the linear span of $S$. Hence, we have
\[
\Delta^{m-1} \times \Delta^{n-1} \contract S = \{ \pi(e_i \times f_j) : e_i \times f_j \notin S \}.
\]
Moreover, if $C_r$ consists of a single edge for all $r$, then $\pi$ is one-to-one on $\Delta^{m-1} \times \Delta^{n-1} \minus S$.

To better understand triangulations of $\Delta^{m-1} \times \Delta^{n-1} \contract S$, let $I$ and $J$ be as above, and construct simplices
\begin{align*}
\Delta_I &:= \{\bar e_1, \dotsc, \bar e_s\} \\
\Delta_J&:= \{\bar f_1, \dotsc, \bar f_t\}.
\end{align*}
By the above computation, the simplification of $\Delta^{m-1} \times \Delta^{n-1} \contract S$ is equivalent to
\[
\Delta_I \times \Delta_J \contract \{\bar e_{p_1} \times \bar f_{q_1}, \dotsc, \bar e_{p_k} \times \bar f_{q_k} \}.
\]
It follows that every triangulation of $\Delta^{m-1} \times \Delta^{n-1} \contract S$ gives a local triangulation of $\Delta_I \times \Delta_J$ at $\{\bar e_{p_1} \times \bar f_{q_1}, \dotsc, \bar e_{p_k} \times \bar f_{q_k} \}$.

To put this more precisely, let $\phi_I : \Delta^{m-1} \to \Delta_I$ and $\phi_J : \Delta^{m-1} \to \Delta_J$ be the maps
\[
\phi_I(e_i) = \bar e_r\ \text{ if }\ e_i \in I_r
\qquad
\phi_J(f_j) = \bar f_r\ \text{ if }\ f_j \in J_r
\]
Let $\phi_{I,J} : \Delta^{m-1} \times \Delta^{n-1} \to \Delta_I \times \Delta_J$ be the map $\phi_I \times \phi_J$. Then $\phi_{I,J}$ maps simplices of $\Delta^{m-1} \times \Delta^{n-1}$ containing $\xi$ to simplices of $\Delta_I \times \Delta_J$ containing $\{\bar e_{p_1} \times \bar f_{q_1}, \dotsc, \bar e_{p_k} \times \bar f_{q_k} \}$, and this map preserves proper intersections and codimensions. Suppose that $\ms T$ is a triangulation of $\Delta^{m-1} \times \Delta^{n-1}$ containing $\xi$. Recall that $\ms T(\xi)$ is the subcollection of simplices of $\ms T$ containing $\xi$. Then
\[
\phi_\xi(\ms T) := \{ \phi_{I,J}(\sigma) : \sigma \in \ms T(\xi) \}
\]
is a local triangulation of $ \Delta_I \times \Delta_J$ at $\{\bar e_{p_1} \times \bar f_{q_1}, \dotsc, \bar e_{p_k} \times \bar f_{q_k} \}$. Moreover, $\phi_{I,J}$ gives a bijection between $\ms T(\xi)$ and $\phi_\xi(\ms T)$.
If $\ms T$ is a local triangulation at $\zeta$ and $\xi \cup \zeta \in \ms T$, then $\phi_\xi(\ms T)$ is a local triangulation of $\Delta_I \times \Delta_J$ at $\{\bar e_{p_1} \times \bar f_{q_1}, \dotsc, \bar e_{p_k} \times \bar f_{q_k} \} \cup \phi_{I,J}(\zeta)$, and $\phi_{I,J}$ gives a bijection between $\ms T(\xi)$ and $\phi_\xi(\ms T)$.

\section{Tools}

With the general theory established, we now put them to use to develop the machinery needed in the proof of the theorem.

\subsection{Orders defined by triangulations}

In this section we will define several orders given by a triangulation $\ms T$. Our main tools will be the restriction and contraction operations defined earlier and the characterization of triangulations of $\Delta^1 \times \Delta^{n-1}$.

\subsubsection{The restriction order $\le_{i_1i_2}$}

Let $i_1$, $i_2 \in [m]$ be distinct, and let $I = \{e_{i_1},e_{i_2}\}$. Let $\ms T$ be a local triangulation of $\Delta^{m-1} \times \Delta^{n-1}$ at $\xi \subseteq I \times \Delta^{n-1}$. The restriction $\ms T[I \times \Delta^{n-1}]$ is a local triangulation of $I \times \Delta^{n-1}$ at $\xi$. By mapping $e_{i_1} \mapsto e_1$ and $e_{i_2} \mapsto e_2$, we obtain a local triangulation of $\Delta^1 \times \Delta^{n-1}$. This triangulation and the order \eqref{vertexorder} give a quasiorder on $\Delta^{n-1}$. We denote this order by $\le_{\ms T[i_1i_2]}$, or simply $\le_{i_1i_2}$ if $\ms T$ is understood. If $\ms T$ is a triangulation, then $\le_{i_1i_2}$ is a total order.

We have the following immediate consequence of Proposition~\ref{compare12}.

\begin{prop} \label{compare}
Let $f_j$, $f_{j'} \in \Delta^{m-1}$ be distinct. Then $f_j <_{i_1i_2} f_{j'}$ if and only if
\[
\{e_{i_2} \times f_{j}, e_{i_1} \times f_{j'}\} \subseteq \sigma \in \ms T.
\]
for some $\sigma$.
\end{prop}

\subsubsection{The partial order $\preceq_i$}

The next order we will construct is a partial order on the set of all maximal simplices of $\Delta^{m-1} \times \Delta^{n-1}$. It tells us if we can go from one maximal simplex to another along a path that always ``tends toward'' a specific direction.

Suppose $\tau$, $\tau'$, are two adjacent maximal simplices of $\Delta^{m-1} \times \Delta^{n-1}$. We can write
\begin{align*}
\tau &= \sigma \cup \{e_i \times f_j\} \\
\tau' &= \sigma \cup \{e_{i'} \times f_{j'}\}
\end{align*}
where $\sigma$ is the common facet of $\tau$, $\tau'$. Since $\sigma$ has codimension 1 and is not contained in a face of $\Delta^{m-1} \times \Delta^{n-1}$, $T(\sigma)$ has exactly two connected components, both of which contain an edge. Let $I = (I_1,I_2)$ and $J=(J_1,J_2)$ be partitions associated to $\sigma$ (as defined in Section~\ref{sec:rescon}) so that $I_1 \cup J_1$ and $I_2 \cup J_2$ are the vertex sets of the components of $T(\sigma)$. Applying equation~\eqref{adjacent} to the adjacent simplices $\phi_{I,J}(\tau)$ and $\phi_{I,J}(\tau')$ in $\Delta_I \times \Delta_J$, we have, without loss of generality, that
\begin{align*}
e_i \times f_j &\in I_1 \times J_2 \\
e_{i'} \times f_{j'} &\in I_2 \times J_1
\end{align*}
If this is the case, then we write
\[
\tau \xrightarrow{I_1,I_2} \tau'.
\]

Now, fix some $i \in [m]$. If $\tau$, $\tau'$ are any two maximal simplices of $\Delta^{m-1} \times \Delta^{n-1}$, we write $\tau \preceq_i \tau'$ if there is a sequence of maximal simplices $\tau_1$, \dots, $\tau_N$ such that
\[
\tau = \tau_1 \xrightarrow{I_1^1,I_2^1} \tau_2 \xrightarrow{I_1^2,I_2^2} \dotsb \xrightarrow{I_1^{N-1},I_2^{N-1}} \tau_N = \tau'
\]
for some partitions $(I_1^r,I_2^r)$ such that $e_i \in I_2^r$ for all $r = 1$, \dots, $N-1$.

\begin{prop} \label{prec1}
The relation $\preceq_i$ is a partial order over the set of maximal simplices of $\Delta^{m-1} \times \Delta^{n-1}$.
\end{prop}

\begin{proof}
We induct on $m$. If $m=1$, then $\Delta^0 \times \Delta^{n-1}$ has only one maximal simplex and the result follows. Suppose $m > 1$. We first note the following.

\begin{prop} \label{monotone}
Suppose $\tau \xrightarrow{I_1,I_2} \tau'$ and let $e_i \in I_2$. Then $N_\tau(e_i) \subseteq N_{\tau'}(e_i)$. Moreover, for any $f_j \in N_\tau(e_i)$, $N_\tau(f_j) \supseteq N_{\tau'}(f_j)$.
\end{prop}

\begin{proof}
This follows easily from the definition of $\xrightarrow{I_1,I_2}$.
\end{proof}

Now, suppose the relation $\preceq_i$ contains a cycle. Thus there exists $\tau_1$, \dots, $\tau_N$, where $N > 1$, such that
\[
\tau_1 \xrightarrow{I_1^1,I_2^1} \tau_2 \xrightarrow{I_1^2,I_2^2} \dotsb \xrightarrow{I_1^{N-1},I_2^{N-1}} \tau_N \xrightarrow{I_1^N,I_2^N} \tau_1
\]
and $e_i \in I^r_2$ for all $r$. Let $f_j \in N_{\tau_1}(e_i)$ be such that $\abs{N_{\tau_1}(f_j)} > 1$; since $T(\tau_1)$ is a spanning tree and $m > 1$, such an $f_j$ must exist. Let $e_{i'} \in N_{\tau_1}(f_j) \minus \{e_i\}$. Since $\tau_1$, \dots, $\tau_N$ is a cycle in $\preceq_i$, Proposition~\ref{monotone} implies that $f_j \in N_{\tau_r}(e_i)$ and $e_{i'} \in N_{\tau_r}(f_j)$ for $r = 1$, \dots, $N$. Hence, $\xi := \{e_i \times f_j, e_{i'} \times f_j\} \subseteq \tau_r$ for all $r$. Thus, we have a cycle
\[
\phi_{I, J}(\tau_1) \xrightarrow{\bar I_1^1,\bar I_2^1} \phi_{I, J}(\tau_2) \xrightarrow{\bar I_1^2,\bar I_2^2} \dotsb \xrightarrow{\bar I_1^{N-1},\bar I_2^{N-1}} \phi_{I, J}(\tau_N) \xrightarrow{\bar I_1^N,\bar I_2^N} \phi_{I, J}(\tau_1)
\]
in $\Delta_{I} \times \Delta_{J}$, where $I$ and $J$ are partitions associated to $\xi$ and $\bar I^r_b := \phi_{I}(I^r_b)$. Since $\phi_{I}(e_i) \in \bar I^r_2$ for all $r$ and $\abs{I} = m-1$, this contradicts the inductive hypothesis, completing the proof.
\end{proof}

\subsubsection{The quasiorder $\preceq_{i_1i_2}$}

We now come to the order which is the main idea of our proof. Like the previous order, it is defined on maximal simplices, and a simplex $\tau'$ comes after $\tau$ in the order if we can go from $\tau$ to $\tau'$ along a path that tends toward a specific direction. However, in this order we also allow this path to move freely back and forth along another specified direction. The result is a quasiorder on the set of maximal simplices of a triangulation.

Let $i_1$, $i_2 \in [m]$ be distinct. Let $\ms T$ be a local triangulation of $\Delta^{m-1} \times \Delta^{n-1}$ at $\{e_{i_2}\} \times J_0$ for some (possibly empty) $J_0 \subseteq \Delta^{n-1}$. Given two maximal simplices $\tau$, $\tau' \in \ms T^\ast$, we write $\tau \preceq_{i_1i_2} \tau'$ if there is a sequence $\tau_1$, \dots, $\tau_N \in \ms T^\ast$ such that
\[
\tau = \tau_1 \xrightarrow{I_1^1,I_2^1} \tau_2 \xrightarrow{I_1^2,I_2^2} \dotsb \xrightarrow{I_1^{N-1},I_2^{N-1}} \tau_N = \tau'
\]
for some partitions $(I_1^r,I_2^r)$ such that for each $r = 1$, \dots, $N-1$, either
\begin{enumerate} \renewcommand{\labelenumi}{(\Alph{enumi})}
\item $e_{i_2} \in I^r_2$, or
\item $\{ I^r_1, I^r_2\} = \{\{e_{i_1}\}, \Delta^{n-1} \minus \{e_{i_1}\}\}$.
\end{enumerate}
Then $\preceq_{i_1i_2}$ is a quasiorder on $\ms T^\ast$. We wish to determine the elements which are equivalent under $\preceq_{i_1i_2}$. We denote this equivalence by $\sim_{i_1i_2}$.

\begin{lem} \label{prec12}
Let $\tau$, $\tau' \in \ms T^\ast$. We have $\tau \sim_{i_1i_2} \tau'$ if and only if there is $\sigma \in \tau \cap \tau'$ such that $T(\sigma)$ has a connected component containing $\Delta^{m-1} \minus \{e_{i_1}\}$.
\end{lem}

\begin{proof}
First assume that such a $\sigma$ exists. Let $I$, $J$ be partitions associated to $\sigma$ such that $I = (\{e_{i_1}\}, \Delta^{m-1} \minus \{e_{i_1}\})$. Then $\phi_{I,J}(\tau)$ and $\phi_{I,J}(\tau')$ are both maximal simplices in $\phi_\sigma(\ms T)$. By Proposition~\ref{m=1}, we thus have a sequence
\[
\phi_{I,J}(\tau) = \tau_1 \xrightarrow{I_1^1,I_2^1} \tau_2 \xrightarrow{I_1^2,I_2^2} \dotsb \xrightarrow{I_1^{N-1},I_2^{N-1}} \tau_N = \phi_{I,J}(\tau')
\]
where $\tau_1$, \dots, $\tau_n \in \phi_\sigma(\ms T)$ and $\{I_1^r,I_2^r\} = \{\{\bar e_1\}, \{\bar e_2\}\}$ for all $r$. Lifting this sequence to $\ms T$, we obtain a sequence of adjacent maximal simplices where each adjacency has the form (B). Reversing this sequence also gives a sequence of this form. Hence $\tau \sim_{i_1i_2} \tau'$.

We now prove the converse. Suppose that $\tau_1$, \dots, $\tau_N \in \ms T^\ast$ are such that
\begin{equation} \label{eq:cycle}
\tau_1 \xrightarrow{I_1^1,I_2^1} \tau_2 \xrightarrow{I_1^2,I_2^2} \dotsb \xrightarrow{I_1^{N-1},I_2^{N-1}} \tau_N \xrightarrow{I_1^N,I_2^N} \tau_1
\end{equation}
where each adjacency satisfies (A) or (B). We need to prove that there is some $\sigma$ contained in all the $\tau_r$ which satisfies the conclusion in the Lemma. We induct on $m$. If $m=2$, by Proposition~\ref{m=1} there is some $f_j$ (namely, the label of the first simplex in the order), such that $e_{i_2} \times f_j \in \tau$ for any $\tau \in \ms T^\ast$. We can thus take $\sigma = \{ e_{i_2} \times f_j \}$.

Assume $m > 2$. Recall that $\ms T$ gives a quasiorder $\le_{i_1i_2}$ on $\Delta^{n-1}$. Let $\le^\text{l}$, be the ``lexicographic-like'' quasiorder defined on the power set of $\Delta^{n-1}$ as follows. Let $J = \{f_{j_1},\dotsc,f_{j_s}\}$ and $J' = \{f_{j_1'},\dotsc,f_{j_t'}\}$ be subsets of $\Delta^{n-1}$ with $f_{j_1} \le_{i_1i_2} \dotsc \le_{i_1i_2} f_{j_s}$ and $f_{j_1'} \le_{i_1i_2} \dotsc \le_{i_1i_2} f_{j_t'}$. Then $J \le^\text{l} J'$ if either
\begin{enumerate} \renewcommand{\labelenumi}{(\roman{enumi})}
\item $s \ge t$ and $f_{j_p} =_{i_1i_2} f_{j_p'}$ for all $p \le t$, or
\item There is some $q \le s$, $t$ such that $f_{j_p} =_{i_1i_2} f_{j_p'}$ for all $p < q$ and $f_{j_q} <_{i_1i_2} f_{j_q'}$.
\end{enumerate}
Note that this order is the same as the lexicographic order except that a prefix of a set $J$ comes \emph{after} $J$ in this order. We have $J <^\text{l} J'$ if and only if (ii) holds or (i) holds and $s > t$. In particular, if $J \supsetneq J'$, then $J <^\text{l} J'$.

Given a simplex $\sigma \in \ms T$, define the sets
\[
R_\sigma := \{ f_j \in N_\sigma(e_{i_2}) : \abs{N_\sigma(f_j)} > 1 \}
\]
and
\[
S_\sigma := \{ f_j \in N_\sigma(e_{i_2}) : f_j \le_{i_1i_2} f_{j'} \text{ for all } f_{j'} \in R_\sigma \}.
\]
Note that if $\sigma$ is maximal, then $R_\sigma \cap S_\sigma$ is nonempty (specifically, it contains the minimal elements of $R_\sigma$ with respect to $\le_{i_1i_2}$). We will prove the following.

\begin{prop} \label{lmonotone}
Suppose $\tau$, $\tau' \in \ms T^\ast$ such that $\tau \xrightarrow{I_1,I_2} \tau'$ and $(I_1,I_2)$ satisfies (A) or (B). Then $S_\tau \ge^\text{l} S_{\tau'}$. If $S_\tau =^\text{l} S_{\tau'}$, then $R_\tau \cap S_\tau \supseteq R_{\tau'} \cap S_{\tau'}$. Also, for any $f_j \in N_\tau(e_{i_2})$, $N_\tau(f_j) \minus \{e_{i_1}\} \supseteq N_{\tau'}(f_j) \minus \{e_{i_1}\}$.
\end{prop}

\begin{proof}
Let $\tau = \sigma \cup \{e_i \times f_j\}$ and $\tau' = \sigma \cup \{e_{i'} \times f_{j'}\}$. We consider separately the cases where $(I_1,I_2)$ satisfies (A) or (B).

\paragraph{If (A) holds}
We have $R_\sigma \subseteq R_\tau$. Since $N_\sigma(e_{i_2}) = N_\tau(e_{i_2})$, we thus have $S_\sigma \supseteq S_\tau$, and hence $S_\sigma \le^\text{l} S_\tau$.
Now, if $e_{i'} \neq e_{i_2}$, then $S_{\tau'} = S_\sigma$. Otherwise, we have $R_{\tau'} = R_\sigma \cup \{f_{j'}\}$ and $N_{\tau'}(e_{i_2}) = N_\sigma(e_{i_2}) \cup \{f_{j'}\}$. If $f_{j'}$ is not a minimal element of $R_{\tau'}$, then $S_{\tau'} = S_\sigma$. Otherwise, $S_{\tau'}$ consists of $f_{j'}$ and all elements of $S_\sigma$ which are less than or equivalent to $f_{j'}$. In this case, $S_{\tau'} <^\text{l} S_\sigma$. In all cases, $S_{\tau'} \le^\text{l} S_\sigma \le^\text{l} S_\tau$.

Now suppose $S_{\tau'} =^\text{l} S_{\tau}$. By the above argument, we must have $S_\sigma =^\text{l} S_\tau$, which holds above only if $S_\sigma = S_\tau$. Since $R_\sigma \subseteq R_\tau$, we thus have $R_\sigma \cap S_\sigma \subseteq R_\tau \cap S_\tau$. We must also have $S_{\tau'} =^\text{l} S_\sigma$, and so either $e_{i'} \neq e_{i_2}$ or $f_{j'}$ is not a minimal element of $R_{\tau'}$. In either case, $R_{\tau'}$ and $R_\sigma$ have the same minimal elements, so $R_{\tau'} \cap S_{\tau'} = R_\sigma \cap S_\sigma$. Thus, $R_{\tau'} \cap S_{\tau'} \subseteq R_\tau \cap S_\tau$. Also, by Proposition~\ref{monotone}, for any $f_{j''} \in N_\tau(e_{i_2})$ we have $N_\tau(f_{j''}) \supseteq N_{\tau'}(f_{j''})$, and thus $N_\tau(f_{j''}) \minus \{e_{i_1}\} \supseteq N_{\tau'}(f_{j''}) \minus \{e_{i_1}\}$.

\paragraph{If (B) holds}
If $(I_1,I_2) = (\{e_{i_1}\}, \Delta^{n-1} \minus \{e_{i_1}\})$, then we have case (A). So assume $(I_1,I_2) = (\Delta^{n-1} \minus \{e_{i_1}\},\{e_{i_1}\})$. We consider two cases.

\emph{Case 1:}
$R_\sigma \neq \emptyset$. Let $f_{j^\ast} \in R_\sigma$. Since
\[
\{e_{i_2} \times f_{j^\ast}, e_{i_1} \times f_j\}, \{e_{i_2} \times f_{j^\ast}, e_{i_1} \times f_{j'}\} \subseteq \tau',
\]
by Proposition~\ref{compare} we have $f_{j^\ast} <_{i_1i_2} f_j$, $f_{j'}$. Since $f_{j^\ast} <_{i_1i_2} f_j$, removing $f_j$ from $N_\tau(e_{i_2})$ and $R_\tau$ does not change $S_\tau$ or the set of minimal elements of $R_\tau$. Similarly, adding $f_{j'}$ to $R_\sigma$ does not change $S_\sigma$ or the set of minimal elements of $R_\sigma$. Thus, $S_\tau = S_\sigma = S_{\tau'}$ and $R_\tau \cap S_\tau = R_\sigma \cap S_\sigma = R_{\tau'} \cap S_{\tau'}$.

\emph{Case 2:}
$R_\sigma = \emptyset$. Then we must have $R_\tau = \{f_j\}$ and $R_{\tau'} = \{f_{j'}\}$. Then since
\[
\{e_{i_2} \times f_{j'}, e_{i_1} \times f_j\} \subseteq \tau',
\]
by Proposition~\ref{compare} we have $f_{j'} <_{i_1i_2} f_j$. It follows that $S_{\tau'}$ consists of $f_{j'}$ and all elements of $S_\tau$ which are less than or equivalent to $f_{j'}$. Thus $S_{\tau'} <^\text{l} S_\tau$.

In either case, it is clear that $N_\tau(f_{j''}) \minus \{e_{i_1}\} \supseteq N_{\sigma}(f_{j''}) \minus \{e_{i_1}\}$ for all $f_{j''} \in \Delta^{n-1}$. This completes the proof of the Proposition.
\end{proof}

We return to the proof of the Lemma. Recall that we have a cycle \eqref{eq:cycle}. By Proposition~\ref{lmonotone}, we must have $S_{\tau_1} =^\text{l} \dotsb =^\text{l} S_{\tau_N}$, and thus $R_{\tau_1} \cap S_{\tau_1} = \dotsb = R_{\tau_N} \cap S_{\tau_N}$. Let $f_j \in R_{\tau_1} \cap S_{\tau_1}$. By Proposition~\ref{lmonotone}, $N_{\tau_r}(f_j) \minus \{e_{i_1}\}$ is the same set $N$ for all $r$.

Clearly $e_{i_2} \in N$. First suppose $N = \{e_{i_2}\}$. Since $f_j \in R_{\tau_r}$ for all $r$, we must have $N_{\tau_r}(f_j) = \{e_{i_1},e_{i_2}\}$ for all $r$. Hence, none of the adjacencies in \eqref{eq:cycle} can satisfy (B). Thus all of them satisfy (A), which means we have a cycle in the order $\preceq_{i_2}$. By Proposition~\ref{prec1}, this means the cycle has one element, in which case the result trivially holds.

Now suppose $N \supsetneq \{e_{i_2}\}$. Then there is some $e_i \neq e_{i_1},e_{i_2}$ which is in $N_{\tau_r}(f_j)$ for all $r$. Thus $\xi := \{e_i \times f_j, e_{i_2} \times f_j\} \subseteq \tau_r$ for all $r$. We obtain a cycle
\[
\phi_{I, J}(\tau_1) \xrightarrow{\bar I_1^1,\bar I_2^1} \phi_{I, J}(\tau_2) \xrightarrow{\bar I_1^2,\bar I_2^2} \dotsb \xrightarrow{\bar I_1^{N-1},\bar I_2^{N-1}} \phi_{I, J}(\tau_N) \xrightarrow{\bar I_1^N,\bar I_2^N} \phi_{I, J}(\tau_1)
\]
where $I$ and $J$ are partitions associated to $\xi$ and $\bar I^r_b := \phi_{I}(I^r_b)$. Each simplex in this cycle belongs to the local triangulation $\phi_\xi(\ms T)$ of $\Delta_{I} \times \Delta_{J}$ at $\{\phi_{I}(e_{i_2})\} \times \phi_{J}(J_0 \cup \{f_j\})$. Furthermore, all adjacencies satisfy (A) or (B) with $e_{i_1}$ and $e_{i_2}$ replaced by $\phi_I(e_{i_1})$ and $\phi_J(e_{i_2})$, respectively. Thus, by the inductive hypothesis, there is some $\bar \sigma$ contained in all the $\phi_{I,J}(\tau_r)$ such that $T(\bar\sigma)$ has a connected component containing $\Delta_I \minus \{\phi_I(e_{i_1})\}$. We can write $\bar \sigma = \phi_{I,J}(\sigma)$ for some $\sigma \in \ms T$ containing $\xi$. This $\sigma$ satisfies the conclusion of the Lemma.
\end{proof}

\subsection{Structure of local triangulations}

Let $\ms T$ be a local triangulation of $\Delta^{m-1} \times \Delta^{n-1}$ at $\xi$. Let $e_{i_0} \times f_{j_0} \in \xi$. The goal of this section is to prove the following.

\begin{prop} \label{uniqueminimal}
There is a unique minimal element of $\ms T^\ast$ with respect to $\preceq_{i_0}$.
\end{prop}

Let $u$, $v$ be vertices of $G_{m,n}$. An \emph{alternating path} (with respect to $\xi$) from $u$ to $v$ is a path in $G_{m,n}$ from $u$ to $v$ such that every other edge of the path is an element of $\xi$. A \emph{1-alternating path} is an alternating path such that the first, third, fifth, and so on edges are in $\xi$, and a \emph{2-alternating path} is an alternating path such that the second, fourth, sixth, and so on edges are in $\xi$. We note the following.

\begin{prop} \label{alternating}
Let $\sigma$, $\sigma' \in \ms T$, and let $b = 1$ or 2. Suppose there are $b$-alternating paths $P$, $P'$ from $u$ to $v$ in $T(\sigma)$, $T(\sigma')$, respectively. Then $P = P'$. 
\end{prop}

\begin{proof}
By walking along $P$ from $u$ to $v$ and then walking backward along $P'$ from $v$ to $u$, we obtain a closed walk. Since $G_{m,n}$ is bipartite and elements of $\xi$ are elements of both $\sigma$ and $\sigma'$, this walk alternates between elements of $\sigma$ and elements of $\sigma'$. If $P \neq P'$, then from this walk we can obtain a cycle which alternates between elements of $\sigma$ and elements of $\sigma'$.\footnote{Again, using bipartiteness.} Thus $\sigma$ and $\sigma'$ contain opposite parts of a circuit, which is impossible for a triangulation.
\end{proof}

\begin{prop} \label{elimination}
Let $\tau \in \ms T^\ast$, and suppose $e_i \times f_j \in \tau$ such that $e_i \times f_j \notin \xi$. Let $\sigma = \tau \minus \{e_i \times f_j\}$. If both connected components of $T(\sigma)$ contain an edge, then there is some $\tau' \in \ms T^\ast$ with $\tau' = \sigma \cup \{e_{i'} \times f_{j'}\}$, where $e_{i'}$ and $f_{j'}$ are in different components of $T(\sigma)$ than $e_i$ and $f_j$, respectively.
\end{prop}

\begin{proof}
Since $e_i \times f_j \notin \xi$, $\sigma \in \ms T$. Let $I$, $J$ be partitions associated to $\sigma$ with $e_{i} \in I_1$ and $f_{j} \in J_2$. Then $\phi_\sigma(\ms T)$ is a local triangulation of $\Delta_I \times \Delta_J$ at $\bar \xi := \{\bar e_1 \times \bar f_1, \bar e_2 \times \bar f_2\}$. By Proposition~\ref{m=1}, we have $\bar \xi \cup \{\bar e_2 \times \bar f_1\} \in \phi_\sigma(\ms T)$. Let $\tau' \in \ms T^\ast$ be such that $\phi_{I,J}(\tau') = \bar\xi \cup \{\bar e_2 \times \bar f_1\}$. This $\tau'$ satisfies the desired conclusion.
\end{proof}

\begin{prop} \label{minimal}
Let $\tau \in \ms T^\ast$. Then $\tau$ is minimal in $\ms T^\ast$ with respect to $\preceq_{i_0}$ if and only if for every $e_i \in \Delta^{m-1}$, the path in $T(\tau)$ from $e_{i_0}$ to $e_i$ is 1-alternating.
\end{prop}

\begin{proof}
Suppose $\tau$ is minimal in $\ms T^\ast$ with respect to $\preceq_{i_0}$. Let $e_i \in \Delta^{m-1}$, and let $P$ be the path in $T(\tau)$ from $e_{i_0}$ to $e_i$. Assume $P$ is not 1-alternating. Then $P$ contains consecutive vertices $e_{i'}$, $f_{j'}$ in that order such that $e_{i'} \times f_{j'} \notin \xi$. By Proposition~\ref{elimination}, there is some $\tau' \in \ms T^\ast$ with $\tau' \xrightarrow{I_1,I_2} \tau$ such that $e_{i'}$, $e_{i_0} \in I_2$. This contradicts the minimality of $\tau$ in $\preceq_{i_0}$.

Conversely, suppose that for each $e_i \in \Delta^{m-1}$ the path in $T(\tau)$ from $e_{i_0}$ to $e_i$ is 1-alternating. Suppose that there is some $\tau' \in \ms T^\ast$ with $\tau' \xrightarrow{I_1,I_2} \tau$, where $e_{i_0} \in I_2$. Let $\sigma = \tau \cap \tau'$, and let $e_{i'} \times f_{j'} = \tau \minus \sigma$. Let $e_i \in I_1$, and consider the path in $T(\tau)$ from $e_{i_0}$ to $e_i$. Since this path is 1-alternating and must contain $e_{i'} \times f_{j'}$, we have $e_{i'} \times f_{j'} \in \xi$. But $e_{i'} \times f_{j'} \notin \tau'$, contradicting $\tau' \in \ms T$. Thus $\tau$ is minimal in $\ms T^\ast$ with respect to $\preceq_{i_0}$.
\end{proof}

We can now prove Proposition~\ref{uniqueminimal}.

\begin{proof}[Proof of Proposition~\ref{uniqueminimal}]
Suppose $\tau$, $\tau'$ are minimal in $\ms T^\ast$ with respect to $\preceq_{i_0}$. Let $\sigma$, $\sigma'$ be the minimal subsets of $\tau$, $\tau'$, respectively, such that $T(\sigma)$ and $T(\sigma')$ contain $\Delta^{m-1}$. By Propositions~\ref{minimal} and \ref{alternating}, we have $\sigma = \sigma'$. Now, $\phi_\sigma(\ms T)$ has a single maximal simplex, and both $\tau$ and $\tau'$ must map to this simplex. Hence $\tau = \tau'$, as desired.
\end{proof}

\subsection{Circuits and flips}

In this final subsection, we collect some facts about flips. We first prove the following general fact about flips of point configurations.

\begin{prop} \label{flip}
Let $\ms T$ be a triangulation of a point set $A$ and let $X = \{X^+,X^-\}$ be a circuit in $A$, where $X^+ = \{x_1,\dotsc,x_k\}$. Suppose that $X^- \in \ms T$. Then $\ms T$ has a flip supported on $(X^+,X^-)$ if and only if there is no maximal simplex $\tau \in \ms T(X^-)^\ast$ with $\abs{\tau \cap X} \le \abs{X} - 2$. If $\ms T$ does not have a flip supported on $(X^+,X^-)$ and $X \minus \{x_i\} \in \ms T$ for some $i$, then such a $\tau$ exists with $\tau \cap X = X \minus \{x_i,x_j\}$ for some $j \neq i$.
\end{prop}

\begin{proof}
First, suppose that $\ms T$ has a flip supported on $(X^+,X^-)$ and such a $\tau$ exists. Since $\abs{\tau \minus X} \ge \abs{\tau} - \abs{X} + 2$ and $\tau$ is maximal, $\tau \minus X$ is not in the link of any maximal simplex of $\ms T_X^+$. Thus, in the collection $\ms T'$ defined in Proposition~\ref{flipdef}, we have $\tau \in \ms T'$ but $X^- \notin \ms T'$, contradicting the fact that $\ms T'$ is a triangulation.

Conversely, suppose that $\ms T$ does not have a flip supported on $(X^+,X^-)$. Consider a maximal simplex $\tau \in \ms T^\ast$ containing $X^-$. If $X \minus \{x_i\} \not\in \ms T$ for all $i$, then $\abs{\tau \cap X} \le \abs{X} - 2$ and we are done. Otherwise, if $X \minus \{x_i\} \in \ms T$, then choose $\tau$ so that $X \minus \{x_i\} \subseteq \tau$.

Suppose there is no $\tau' \in \ms T(X^-)^\ast$ and $j \neq i$ such that $\tau' \cap X = X \minus \{x_i,x_j\}$. Let $j \neq i$. Consider the facet $\sigma := \tau \minus \{x_j\}$ of $\tau$. We claim that $\sigma$ is not contained in a face of $\conv(A)$. Assume the contrary, and let $H$ be a supporting hyperplane of this face. Since $\tau$ is a maximal simplex, $x_j \notin H$. Then since $x_i$, $x_j$ are in $X^+$ of the circuit $X$ and $X \minus \{x_i,x_j\} \subseteq H$, this implies $x_i$ and $x_j$ are on opposite sides of $H$. This contradicts the fact that $H$ is a supporting hyperplane. Thus, $\sigma$ is not contained in a face of $\conv(A)$. It follows that there is another maximal simplex $\tau' \in \ms T^\ast$ containing $\sigma$. If $x_i \notin \tau'$, then $\abs{\tau' \cap X} = X \minus \{x_i,x_j\}$, a contradiction. Thus we have that $\tau' = \tau \minus \{x_j\} \cup \{x_i\}$, and hence $X \minus \{x_j\} \in \ms T$ and $\tau \minus X \in \link_{\ms T}(X \minus \{x_j\})$. We thus have $\link_{\ms T}(X \minus \{x_i\}) \subseteq \link_{\ms T}(X \minus \{x_j\})$.

Switching $i$ and $j$ in the above argument, we either have $\tau' \cap X = X \minus \{x_j,x_i\}$ for some $\tau' \in \ms T(X^-)^\ast$ or $\link_{\ms T}(X \minus \{x_j\}) \subseteq \link_{\ms T}(X \minus \{x_i\})$. Hence $\link_{\ms T}(X \minus \{x_i\}) = \link_{\ms T}(X \minus \{x_j\})$. Since this holds for all $j \neq i$, we have $\ms T_X^+ \subseteq \ms T$ and every maximal simplex of $\ms T_X^+$ has the same link in $\ms T$. Hence $\ms T$ has a flip supported on $(X^+,X^-)$, a contradiction. This completes the proof.
\end{proof}

Now, let $\ms T$ be a triangulation of $\Delta^{m-1} \times \Delta^{n-1}$. Let $X = \{X^+,X^-\}$ be a circuit of $\Delta^{m-1} \times \Delta^{n-1}$, where
\begin{equation}
\begin{split} \label{circuitnotation}
X^- &= \{e_{i_1} \times f_{j_1}, e_{i_2} \times f_{j_2}, \dotsc, e_{i_k} \times f_{j_k}\} \\
X^+ &= \{e_{i_2} \times f_{j_1}, e_{i_3} \times f_{j_2}, \dotsc, e_{i_1} \times f_{j_k}\}
\end{split}
\end{equation}
By Proposition~\ref{uniqueminimal}, for each $r = 1$, \dots, $k$, $\ms T(X^-)^\ast$ has a unique minimal element in the order $\preceq_{i_r}$. Denote this element by $\tau_r$. For each $r = 1$, \dots, $k$, let
\[
\sigma_r := X \minus \{e_{i_r} \times f_{j_{r-1}}\}
\]
be a maximal simplex of $\ms T_X^+$.\footnote{Here, as in the rest of the section, subscripts of $i$, $j$, and $\sigma$ are taken modulo $k$.}

\begin{prop} \label{minimalcontainssimplex}
If $\sigma_r \in \ms T$, then $\sigma_r \subseteq \tau_r$.
\end{prop}

\begin{proof}
We have $\sigma_r \in \ms T(X^-)$. The edges of $T(\sigma_r \minus \{e_{i_{r-1}} \times f_{j_{r-1}}\})$ form a 1-alternating path with respect to $X^-$ from $e_{i_r}$ to $e_{i_{r+1}}$, $e_{i_{r+2}}$, \dots, $e_{i_{r-1}}$. So by Propositions~\ref{minimal} and \ref{alternating}, $\sigma_r \minus \{e_{i_{r-1}} \times f_{j_{r-1}}\} \subseteq \tau_r$. Finally, $e_{i_{r-1}} \times f_{j_{r-1}} \in \tau_r$ because $X^- \subseteq \tau_r$.
\end{proof}

Now assume that $\ms T$ has a flip supported on $(X^+,X^-)$. Let $\ms T'$ be the result of this flip.
We will prove the following two propositions. They will later be used to determine the effect that certain flips have on the shapes of the simplices involved.

\begin{prop} \label{uniqueattach}
Let $e_i \in \Delta^{m-1}$ where $i \neq i_1$, \dots, $i_k$. Then there is a unique $j \in \{j_1,\dotsc,j_k\}$ such that $\sigma_r \cup \{e_i \times f_j\} \in \ms T$ for some $r$.
\end{prop}

\begin{proof}
We first show that such a $j$ exists. By Proposition~\ref{minimal}, there is a 1-alternating path with respect to $X^-$ from $e_{i_1}$ to $e_i$ in $T(\tau_1)$. Since $i \neq i_1$, \dots, $i_k$, the last edge of this path must be of the form $e_i \times f_j$ for some $j \in \{j_1,\dotsc,j_k\}$. Thus $\{e_i \times f_j\} \subseteq \tau_1$, and so by Proposition~\ref{minimalcontainssimplex}, $\sigma_1 \cup \{e_i \times f_j\} \in \ms T$.

Now suppose a different such $j'$ exists, and let $\sigma_r$ be such that $\sigma_r \cup \{e_i \times f_{j'}\} \in \ms T$. By Proposition~\ref{flipdef}, $\sigma_r$ and $\sigma_1$ have the same link in $\ms T$, so we may assume $\sigma_r = \sigma_1$. Then there is a 1-alternating path from $e_{i_1}$ to $e_i$ in $T(\sigma_1 \cup \{e_i \times f_j\})$ and a different 1-alternating path from $e_{i_1}$ to $e_i$ in $T(\sigma_1 \cup \{e_i \times f_{j'}\})$. This contradicts Proposition~\ref{alternating}. Hence $j$ is unique.
\end{proof}

\begin{prop} \label{shapeshift}
There is a bijection $\Psi: \ms T^\ast \to (\ms T')^\ast$ such that for each $\tau \in \ms T^\ast$ and $f_j \in \Delta^{n-1}$, we have
\begin{enumerate}
\item $N_{\Psi(\tau)}(f_j) = N_\tau(f_j)$ if $j \neq j_1$, \dots, $j_k$.
\item If $j = j_r$, then either $N_{\Psi(\tau)}(f_j) = N_\tau(f_j)$ or
\[
N_{\Psi(\tau)}(f_j) = N_\tau(f_j) \minus \{e_{i_r}\} \cup \{e_{i_{r+1}}\},
\]
with the latter occurring if and only if $\sigma_{r+1} = X \minus \{e_{i_{r+1}} \times f_{j_r}\} \subseteq \tau$.
\end{enumerate}
\end{prop}

\begin{proof}
Let $\tau \in \ms T^\ast$. By Proposition~\ref{flip}, either $\tau$ does not contain $X^-$ or $\tau$ contains a maximal simplex of $\ms T_X^+$. In the former case, define $\Psi(\tau) = \tau$. Otherwise, suppose $\sigma_{r+1} \subseteq \tau$ for some $r$. Define
\[
\Psi(\tau) = \tau \minus \{e_{i_r} \times f_{j_r}\} \cup \{e_{i_{r+1}} \times f_{j_r}\}.
\]
By Proposition~\ref{flipdef}, $\Psi$ is a bijection $\ms T^\ast$ to $(\ms T')^\ast$, and it satisfies the desired properties.
\end{proof}

Finally, we note the following relationship between flips and the restriction order $\le_{i_1i_2}$.

\begin{prop} \label{fliporder}
Let $\ms T$, $X$, and $\ms T'$ be as above. Let $i$, $i'$ be distinct elements of $[m]$. If $\abs{X^-} \neq 2$ or $\abs{X^-} = 2$ and $\{i,i'\} \neq \{i_1,i_2\}$, then the orders $\le_{\ms T[ii']}$ and $\le_{\ms T'[ii']}$ are the same. If $\abs{X^-} = 2$ and $\{i,i'\} = \{i_1,i_2\}$, then the orders $\le_{\ms T[ii']}$ and $\le_{\ms T'[ii']}$ are the same except the order of the consecutive elements $f_{j_1}$, $f_{j_2}$ is flipped.
\end{prop}

\begin{proof}
This follows easily from Proposition~\ref{shapeshift} and Proposition~\ref{compare}.
\end{proof}

\section{The main proof}

We are now ready to prove the main result.

\begin{thm}
The set of triangulations of $\Delta^3 \times \Delta^{n-1}$ is flip-connected.
\end{thm}

Our proof will be an algorithm that starts with any triangulation of $\Delta^3 \times \Delta^{n-1}$ and applies flips to reach a specific triangulation. The algorithm will have three ``phases''. The only difficult phase is Phase I; in terms of the Cayley trick, the purpose of Phase I is to move all unmixed cells of $\mc C(\ms T)$ to the face $n\{e_1,e_2,e_3\}$ of $n\Delta^3$. The purpose of Phase II is to move all the unmixed cells to the edge $n\{e_1,e_2\}$. Phase III then permutes the unmixed cells and sorts out the remaining cells.

For notational convenience, we will now refer to $e_i$ as simply $i$.

\subsection{Phase I}

Let $\ms T$ be a triangulation of $\Delta^3 \times \Delta^{n-1}$. Let $\ms T^\text{I}$ be the set of all maximal simplices $\tau \in \ms T^\ast$ for which there is some $f_j \in \Delta^{n-1}$ with $\{1,2\} \subseteq N_\tau(f_j)$ and $4 \notin N_\tau(f_j)$. The goal of this section is to prove the following.

\begin{claim} \label{phaseIclaim}
$\ms T$ is flip-connected to a triangulation $\ms T'$ with $(\ms T')^\text{I} = \emptyset$.
\end{claim}

Assume $\ms T^\text{I} \neq \emptyset$. To prove the Claim, it suffices to show that $\ms T$ is flip-connected to some $\ms T'$ with $\abs{(\ms T')^\text{I}} < \abs{\ms T^\text{I}}$. We will use the following Propositions to determine how certain flips affect $\ms T^\text{I}$. In all of the below Propositions, $X$ is as in \eqref{circuitnotation}.

\begin{prop} \label{Iflip1}
Suppose that $\ms T$ has a flip supported on $(X^+,X^-)$, and let $\ms T'$ be the result of this flip. Assume that $1 = i_1$ and $2 \notin \{i_1, \dots, i_k\}$. Suppose that there is a maximal simplex $\sigma$ of $\ms T_X^+$ and $j \in \{j_1,\dotsc,j_k\}$ such that $\sigma \cup \{2 \times f_j\} \in \ms T$. We have the following.
\begin{enumerate} \renewcommand{\labelenumi}{(\alph{enumi})}
\item If $j \neq j_k$, then $\abs{(\ms T')^\text{I}} \le \abs{\ms T^\text{I}}$.
\item If $j = j_1$ and $4 \in \{i_1,\dotsc,i_k\}$, then $\abs{(\ms T')^\text{I}} < \abs{\ms T^\text{I}}$.
\end{enumerate}
\end{prop}

\begin{proof}
Let $\Psi$ be as in Proposition~\ref{shapeshift}. Suppose $N_{\Psi(\tau)}(f_{j'}) \supseteq \{1,2\}$ for some $\Psi(\tau) \in \ms T'$ and $f_{j'} \in \Delta^{n-1}$. Then by Proposition~\ref{shapeshift}, either $N_\tau(f_{j'}) = N_{\Psi(\tau)}(f_{j'})$ or $j' = j_k$ and $\sigma_1 \cup \{2 \times f_{j_k}\} \subseteq \tau$. In the latter case, Proposition~\ref{uniqueattach} implies $j = j_k$. Thus if $j \neq j_k$, $\abs{(\ms T')^\text{I}} \le \abs{\ms T^\text{I}}$. If furthermore $j = j_1$, then
\[
\ms S := \{ \tau \in \ms T^\ast : \sigma_2 \cup \{2 \times f_{j_1}\} \subseteq \tau \}
\]
is nonempty. If furthermore $4 \in \{i_1,\dotsc,i_k\}$, then $4 \notin N_\tau(f_{j_1})$ for any $\tau \in \ms S$, so $\ms S \subseteq \ms T^\text{I}$. But $\Psi(\ms S) \cap (\ms T')^\text{I} = \emptyset$, so $\abs{(\ms T')^\text{I}} < \abs{\ms T^\text{I}}$.
\end{proof}

Note that we could swap the roles of 1 and 2 in the above Proposition and the result would still hold. The following Propositions are proved similarly; we leave them to the reader.

\begin{prop} \label{Iflip2}
Suppose that $\ms T$ has a flip supported on $(X^+,X^-)$, and let $\ms T'$ be the result of this flip. Assume that 1, $2 \notin \{i_1, \dots, i_k\}$ and $4 = i_2$. Suppose there is a maximal simplex $\sigma$ of $\ms T_X^+$ and $j$, $j' \in \{j_1,\dotsc,j_k\}$ such that $\sigma \cup \{1 \times f_j, 2 \times f_{j'}\} \in \ms T$. We have the following.
\begin{enumerate} \renewcommand{\labelenumi}{(\alph{enumi})}
\item If $j \neq j'$, then $\abs{(\ms T')^\text{I}} = \abs{\ms T^\text{I}}$.
\item If $j = j' = j_1$, then $\abs{(\ms T')^\text{I}} < \abs{\ms T^\text{I}}$.
\end{enumerate}
\end{prop}

\begin{prop} \label{Iflip3}
Suppose that $\ms T$ has a flip supported on $(X^+,X^-)$, and let $\ms T'$ be the result of this flip. Assume that 1, $2 \in \{i_1,\dotsc,i_k\}$. Then $\abs{(\ms T')^\text{I}} = \abs{\ms T^\text{I}}$.
\end{prop}

Now, let $\tau_\text{I}$ be any maximal element of $\ms T^\text{I}$ with respect to $\preceq_{34}$. Let $\sh(\tau_\text{I}) = \{N_1,\dots,N_k\}$, and without loss of generality, let $f_1$, \dots, $f_k \in \Delta^{n-1}$ be such that $N_r = N_{\tau_\text{I}}(f_r)$. Since $\tau_\text{I} \in \ms T^\text{I}$, we can assume $\{1,2\} \subseteq N_1$ (and thus $4 \notin N_1$). We then have three distinct possibilities for $\sh(\tau_\text{I})$.
\begin{enumerate}
\item $\{1,4\}$ or $\{2,4\} \subseteq N_2$
\item $N_1 = \{1,2\}$, $N_2 = \{3,4\}$, $N_3 = \{1,3\}$ or $\{2,3\}$
\item $N_1 = \{1,2,3\}$, $N_2 = \{3,4\}$
\end{enumerate}
The remainder of Phase I will depend on which case we are in. Each case will use the same strategy, which we outline as follows. Let $\sigma_\text{I}$ be the unique minimal subset of $\tau_\text{I}$ such that $T(\sigma_\text{I})$ has a connected component containing $\{1,2,4\}$. Call a subcollection $\ms S \subseteq \ms T$ of simplices $\tau_\text{I}$-\emph{good} if both of the following hold.
\begin{enumerate} \renewcommand{\labelenumi}{(\alph{enumi})}
\item There is no maximal simplex $\tau \in \ms S$ such that $\tau \in \ms T^\text{I}$ and $\sigma_\text{I} \not\subseteq \tau$.
\item There is no maximal simplex $\tau \in \ms S$ and $j \neq 1$, 2 such that $\{1,2\} \subseteq N_\tau(f_j)$.
\end{enumerate}
The strategy will be to define a circuit $X = \{X^+,X^-\}$ such that $\ms T(X^-)$ is $\tau_\text{I}$-good, and so that if $\ms T$ has a flip supported on $(X^+,X^-)$, then the result will decrease the size of $\ms T^\text{I}$. Once $X$ is defined, we will use $\tau_\text{I}$-goodness to find a series of intermediate flips starting from $\ms T$ so that the final result has a flip supported on $(X^+,X^-)$.

\subsubsection{Case 1: $\{1,4\}$ or $\{2,4\} \subseteq N_2$}

We will assume $\{1,4\} \subseteq N_2$; the other case follows analogously. Let $X = \{X^+,X^-\}$ be the circuit
\begin{align*}
X^- &= \{ 1 \times f_1, 4 \times f_2 \} \\
X^+ &= \{ 4 \times f_1, 1 \times f_2 \}
\end{align*}
Then $\sigma = X \minus \{4 \times f_1\}$ is a maximal simplex of $\ms T_X^+$, and $\sigma_\text{I} = \sigma \cup \{2 \times f_1\} \subseteq \tau_\text{I}$. Thus, if $(X^+,X^-)$ supports a flip of $\ms T$, then this flip satisfies the conditions of Proposition~\ref{Iflip1}(b), and we would be done.

Let $\tau_\ast$ be the unique minimal element of $\ms T(X^-)^\ast$ with respect to $\preceq_4$. Note that $\sigma_\text{I} \in \ms T(X^-)$, and the edges of $T(\sigma_\text{I})$ give a 1-alternating path with respect to $X^-$ from 4 to 1 and 2. Thus, by Propositions~\ref{minimal} and \ref{alternating}, we have $\sigma_\text{I} \subseteq \tau_\ast$. By Lemma~\ref{prec12}, it follows that $\tau_\ast \sim_{34} \tau_\text{I}$. Thus, $\tau_\ast$ is maximal in $\preceq_{34}$. Moreover, its shape satisfies Case 1. We may thus redefine $\tau_\text{I}$ as $\tau_\text{I} = \tau_\ast$.

We can now make the following key observation.

\begin{prop} \label{good1}
$\ms T(X^-)$ is $\tau_\text{I}$-good.
\end{prop}

\begin{proof}
First, suppose there is some $\tau \in \ms T(X^-)^\ast$ such that $\tau \in \ms T^\text{I}$ and $\sigma_\text{I} \not\subseteq \tau$. Since $\tau_\text{I}$ is minimal in $\ms T(X^-)^\ast$ with respect to $\preceq_4$, we have $\tau_\text{I} \preceq_{34} \tau$. On the other hand, $\sigma_\text{I} \not\subseteq \tau$, and $\sigma_\text{I}$ is the unique minimal subset of $\tau_\text{I}$ such that $T(\sigma_\text{I})$ has a connected component containing $\{1,2,4\}$. Thus, by Proposition~\ref{prec12}, $\tau_\text{I} \not\sim_{34} \tau$. Hence $\tau_\text{I} \prec_{34} \tau$, which contradicts the maximality of $\tau_\text{I}$ in $\ms T^\text{I}$ with respect to $\preceq_{34}$.

Now suppose there is some $\tau \in \ms T(X^-)^\ast$ and $j \neq 1$, 2 such that $\{1,2\} \subseteq N_{\tau}(f_j)$. Since $j \neq 2$, by Proposition~\ref{elimination}, there is some $\tau' \in \ms T(X^-)$ (possibly the same as $\tau$) with $N_{\tau'}(f_j) = N_{\tau}(f_j) \minus \{4\}$. Since $\{1,2\} \subseteq N_{\tau'}(f_j)$ and $j \neq 1$, we must also have $\sigma_\text{I} \not\subseteq \tau'$. Thus $\tau' \in \ms T^\text{I}$ and $\sigma_\text{I} \not\subseteq \tau'$, which as above is a contradiction.
\end{proof}

It now suffices to prove the following claim.

\begin{claim} \label{case1claim}
Let $\ms T$ be any triangulation of $\Delta^3 \times \Delta^{n-1}$ such that $\tau_\text{I} \in \ms T$ and $\ms T(X^-)$ is $\tau_\text{I}$-good. Then either there is a flip of $\ms T$ supported on $(X^+,X^-)$, or there is a flip of $\ms T$ with result $\ms T'$ such that $\tau_\text{I} \in \ms T'$, $\ms T'(X^-)$ is $\tau_\text{I}$-good, $\abs{(\ms T')^\text{I}} \le \abs{\ms T^\text{I}}$, and $\abs{\ms T'(X^-)^\ast} < \abs{\ms T(X^-)^\ast}$.
\end{claim}

By repeatedly applying this Claim to our original $\ms T$, we eventually obtain that $\ms T$ is flip-connected to some $\ms T'$ for which $\tau_\text{I} \in \ms T'$, $\abs{(\ms T')^\text{I}} \le \abs{\ms T^\text{I}}$, and $\ms T'$ has a flip supported on $(X^+,X^-)$. Then by Proposition~\ref{Iflip1}, the result $\ms T''$ of this flip satisfies $\abs{(\ms T'')^\text{I}} < \abs{\ms T^\text{I}}$, as desired.

\begin{proof}[Proof of Claim~\ref{case1claim}]

Let $\ms S$ be the set of all maximal simplices $\tau \in \ms T(X^-)^\ast$ such that $\tau \cap X^+ = \emptyset$. If $\ms S = \emptyset$, then by Proposition~\ref{flip}, $\ms T$ has a flip supported on $(X^+,X^-)$, and we are done.
Assume $\ms S \neq \emptyset$.

We will use the following to restrict the possible shapes of simplices in $\ms S$.

\begin{prop} \label{restrictions}
Let $\tau \in \ms S$. The following are true.
\begin{multicols}{2}
\begin{enumerate} \renewcommand{\labelenumi}{(\alph{enumi})}
\item $1 \times f_2 \notin \tau$
\item $4 \times f_1 \notin \tau$
\item $2 \times f_1 \notin \tau$
\item $2 \times f_2 \notin \tau$
\end{enumerate}
\end{multicols}
\end{prop}

\begin{proof}
Parts (a) and (b) hold by the definition of $\ms S$. Suppose $2 \times f_1 \in \tau$. By part (b), we have $4 \notin N_{\tau}(f_1)$. Hence $\tau \in \ms T^\text{I}$. By part (a), we have $\sigma_\text{I} \not\subseteq \tau$. This contradicts $\tau_\text{I}$-goodness(a), proving (c).

Finally, since $\tau_\text{I}$ is minimal in $\ms T(X^-)^\ast$ with respect to $\preceq_4$, by Proposition~\ref{monotone} we have $N_{\tau_\text{I}}(f_2) \supseteq N_\tau(f_2)$. Hence $2 \times f_2 \notin \tau$, proving (d).
\end{proof}

Suppose $\tau \in \ms S$. Let
\[
P = \{ 1 \times f_{j_1}, i_1 \times f_{j_1}, i_1 \times f_{j_2}, i_2 \times f_{j_2}, \dotsc, i_{k-1} \times f_{j_k}, 4 \times f_{j_k} \}
\]
be the path in $T(\tau)$ from 1 to 4. Suppose first that
\[
P = \{1 \times f_{j_1}, 4 \times f_{j_1}\}.
\]
By Proposition~\ref{restrictions}(a) and (b), we have $j_1 \neq 1$, 2. Consider the minimal element $\tau'$ of $\ms T(X^- \cup P)^\ast$ with respect to $\preceq_4$. Since $P \subseteq \tau'$, we have $\tau' \in \ms S$. Now, by Proposition~\ref{minimal}, we must have $2 \times f_j \in \tau'$ for some $j = 1$, 2, or $j_1$. This contradicts either Proposition~\ref{restrictions}(c), (d), or $\tau_\text{I}$-goodness(b). Hence $P$ cannot be of this form.

Furthermore, if $i_1 = 2$, then we have $\{1,2\} \subseteq N_\tau(f_{j_1})$, which contradicts either Proposition~\ref{restrictions}(c) or $\tau_\text{I}$-goodness(b). Hence, $i_1 = 3$.

Now suppose that $f_{j_k} = f_2$. By Proposition~\ref{restrictions}(d), we have $i_{k-1} \neq 2$. Hence $i_{k-1} = 3$, and
\[
P = \{1 \times f_{j_1}, 3 \times f_{j_1}, 3 \times f_2, 4 \times f_2\}
\]
Let $\tau'$ be the minimal element of $\ms T(X^- \cup P)^\ast$ with respect to $\preceq_4$. Since $P \subseteq \tau'$, we have $\tau' \in \ms S$. By Proposition~\ref{minimal}, we have $2 \times f_j \in \tau'$ for some $j = 1$, 2, or $j_1$. As before, this is a contradiction. Hence, $j_k \neq 2$.

We have thus restricted $P$ to two possible forms:
\begin{enumerate} \renewcommand{\labelenumi}{(\roman{enumi})}
\item $P = \{1 \times f_{j_1}, 3 \times f_{j_1}, 3 \times f_{j_2}, 4 \times f_{j_2}\}$, $j_2 \neq 2$
\item $P = \{1 \times f_{j_1}, 3 \times f_{j_1}, 3 \times f_{j_2}, 2 \times f_{j_2}, 2 \times f_{j_3}, 4 \times f_{j_3}\}$, $j_3 \neq 2$.
\end{enumerate}
Now, choose $\tau$ so that it is a maximal element of $\ms S$ with respect to $\preceq_{21}$.\footnote{Here, as in the rest of the proof, we mean the order $\preceq_{21}$ which is defined on all of $\ms T$.} We have two possibilities depending on $P$.

\paragraph{Subcase 1.1: $P$ has form (i)}
Let $Y = \{Y^+,Y^-\}$ be the circuit
\begin{align*}
Y^- &= \{ 1 \times f_{j_1}, 3 \times f_{j_2}, 4 \times f_2 \} \\
Y^+ &= \{3 \times f_{j_1}, 4 \times f_{j_2}, 1 \times f_2 \}
\end{align*}
Let $\rho = Y \minus \{1 \times f_2\}$. Then $\rho$ is a maximal simplex of $\ms T_Y^+$ and $\rho \cup \{1 \times f_1\} = X^- \cup P \subseteq \tau$. Let $\tau_\ast$ be the minimal element of $\ms T(X^- \cup P)^\ast$ with respect to $\preceq_1$. Since $P \subseteq \tau_\ast$, we have $\tau_\ast \in \ms S$.
Now, by Proposition~\ref{minimal}, we must have $2 \times f_j \in \tau_\ast$ for some $j = 1$, 2, $j_1$, or $j_2$. We cannot have $j=1$, 2, or $j_1$ by Propositions~\ref{restrictions}(c), (d), and $\tau_\text{I}$-goodness(b). Thus $2 \times f_{j_2} \in \tau_\ast$. Then by Proposition~\ref{minimal}, it follows that $\tau_\ast$ is the minimal element of $\ms T(Y^-)^\ast$ with respect to $\preceq_1$.

Since $P \subseteq \tau_\ast$, by Lemma~\ref{prec12}, we have $\tau \sim_{21} \tau_\ast$. Hence $\tau_\ast$ is a maximal element of $\ms S$ with respect to $\preceq_{21}$, and $P \subseteq \tau_\ast$. We may thus redefine $\tau$ as $\tau = \tau_\ast$.

We first claim that $\ms T(Y^-) \subseteq \ms T(X^-)$. Suppose that $\tau' \in \ms T(Y^-)^\ast$. Clearly $4 \times f_2 \in \tau'$. Also, since $\tau$ is minimal in $\ms T(Y^-)^\ast$ with respect to $\preceq_1$, by Proposition~\ref{monotone} we have $N_\tau(1) \subseteq N_{\tau'}(1)$. Hence $1 \times f_1 \in \tau'$. So $\tau' \in \ms T(X^-)$, as desired.

We now claim that $\ms T$ has a flip supported on $(Y^+,Y^-)$. Suppose the contrary. Since $\rho \in \ms T$, by Proposition~\ref{flip}, there is some $\tau' \in \ms T(Y^-)^\ast$ such that $\abs{\tau' \cap Y^+} \le \abs{Y^+}-2 = 1$ and $1 \times f_2 \notin \tau'$. As just shown, $\tau' \in \ms T(X^-)$. Moreover, by Proposition~\ref{monotone}, we have $N_{\tau}(f_1) \supseteq N_{\tau'}(f_1)$, so $4 \times f_1 \notin \tau'$. Hence $\tau' \in \ms S$. Since $\abs{\tau' \cap Y^+} \le 1$, we have $P \not\subseteq \tau'$, and hence by Lemma~\ref{prec12}, $\tau \not\sim_{21} \tau'$. Since $\tau$ is minimal in $\ms T(Y^-)$ with respect to $\preceq_1$, we thus have $\tau \prec_{21} \tau'$. This contradicts the maximality of $\tau$ in $\ms S$ with respect to $\preceq_{21}$. Hence, $\ms T$ has a flip supported on $(Y^+,Y^-)$.

Let $\ms T'$ be the result of this flip. By Proposition~\ref{shapeshift}, if $j_1 \neq 1$, we have
\[
\ms T'(X^-)^\ast = \ms T(X^-)^\ast \minus \ms T(Y^-)^\ast \cup \Psi( \ms T(X \minus \{3 \times f_{j_1}\})^\ast) \cup \Psi( \ms T(X \minus \{4 \times f_{j_2}\})^\ast).
\]
and if $j_1 = 1$, we have
\[
\ms T'(X^-)^\ast = \ms T(X^-)^\ast \minus \ms T(Y^-)^\ast \cup \Psi( \ms T(X \minus \{4 \times f_{j_2}\})^\ast).
\]
Either way, $\abs{\ms T'(X^-)^\ast} < \abs{\ms T(X^-)^\ast}$. By Proposition~\ref{Iflip1}(a), since $\rho \cup \{2 \times f_{j_2}\} \in \ms T(X^-)$, we have $\abs{(\ms T')^\text{I}} \le \abs{\ms T^\text{I}}$. It is easy to check using the above equations that if $\ms T(X^-)$ is $\tau_\text{I}$-good, then so is $\ms T'(X^-)$. Finally, to show that $\tau_\text{I} \in \ms T'(X^-)$, it suffices to show $\tau_\text{I} \notin \ms T(Y^-)$. If $\tau_\text{I} \in \ms T(Y^-)$, then by Proposition~\ref{monotone}, $N_\tau(f_1) \supseteq N_{\tau_\text{I}}(f_1)$, so $2 \times f_1 \notin \tau_\text{I}$, a contradiction. Hence $\tau_\text{I} \notin \ms T(Y^-)$, which completes the Subcase.

\paragraph{Subcase 1.2: $P$ has form (ii)}

Let $Y = \{Y^+,Y^-\}$ be the circuit
\begin{align*}
Y^- = \{ 1 \times f_{j_1}, 3 \times f_{j_2}, 2 \times f_{j_3}, 4 \times f_2 \} \\
Y^+ = \{3 \times f_{j_1}, 2 \times f_{j_2}, 4 \times f_{j_3}, 1 \times f_2 \}
\end{align*}
Let $\rho = Y \minus \{1 \times f_2\}$. Then $\rho$ is a maximal simplex of $\ms T_Y^+$ and $\rho \subseteq \tau$. By Proposition~\ref{minimal}, $\tau$ is the minimal element of $\ms T(Y^-)^\ast$ with respect to $\preceq_1$.

By the same arguments as in the previous subcase, we have $\ms T(Y^-) \subseteq \ms T(X^-)$ and $\ms T$ has a flip supported on $(Y^+,Y^-)$. If the result of this flip is $\ms T'$, then Proposition~\ref{Iflip3} implies $\abs{(\ms T')^\text{I}} = \abs{\ms T^\text{I}}$. Similar arguments to the previous subcase show that the rest of Claim~\ref{case1claim} holds. This concludes Case 1.
\end{proof}

\subsubsection{Case 2: $N_1 = \{1,2\}$, $N_2 = \{3,4\}$, $N_3 = \{1,3\}$ or $\{2,3\}$}
We will assume $N_2 = \{1,3\}$; the other case follows analogously. Let $X = \{X^+,X^-\}$ be the circuit
\begin{align*}
X^- &= \{ 1 \times f_1, 4 \times f_2, 3 \times f_3 \} \\
X^+&= \{ 4 \times f_1, 3 \times f_2, 1 \times f_3 \}
\end{align*}
Then $\sigma = X \minus \{4 \times f_1\}$ is a maximal simplex of $\ms T_X^+$, and $\sigma_\text{I} = \sigma \cup \{2 \times f_1\} \subseteq \tau_\text{I}$. Thus, if $(X^+,X^-)$ supports a flip of $\ms T$, then this flip satisfies the conditions of Proposition~\ref{Iflip1}(b), and we would be done.

By Proposition~\ref{minimal}, $\tau_\text{I}$ is minimal in $\ms T(X^-)^\ast$ with respect to $\preceq_4$. We can use the same argument as in Case 1 to prove that $\ms T(X^-)$ is $\tau_\text{I}$-good.

It now suffices to prove Claim~\ref{case1claim} with $\tau_\text{I}$, $\sigma_\text{I}$, and $X$ redefined as in this Case. Let $\ms S$ be the set of all maximal simplices $\tau \in \ms T(X^-)^\ast$ for which $\abs{\tau \cap X^+} \le 1$ and $4 \times f_1 \notin \tau$. Since $\sigma \in \ms T$, by Proposition~\ref{flip} we have that $(X^+,X^-)$ supports a flip of $\ms T$ if and only if $\ms S = \emptyset$. So assume $\ms S \neq \emptyset$.

\begin{prop} \label{restrictions2}
Let $\tau \in \ms S$. The following are true.
\begin{enumerate} \renewcommand{\labelenumi}{(\alph{enumi})}
\item $1 \times f_2 \notin \tau$
\item $4 \times f_1 \notin \tau$
\item $2 \times f_j \notin \tau$ for $j = 1$, 2, 3.
\end{enumerate}
\end{prop}

\begin{proof}
Suppose $1 \times f_2 \in \tau$. Then, with respect to $X^-$, we have a 2-alternating path $f_2 \rightarrow 1$ in $T(\tau)$ and a 2-alternating path $f_2 \rightarrow 3 \rightarrow f_3 \rightarrow 1$ in $T(\tau_\text{I})$. This contradicts Proposition~\ref{alternating}, proving (a). Part (b) follows from the definition of $\ms S$.

Suppose $2 \times f_j \in \tau$. First suppose $j = 1$. By (b), $4 \times f_1 \notin \tau$, so $\tau \in \ms T^\text{I}$. Since $\abs{\tau \cap X^+} \le 1$, we have $\sigma_\text{I} \not\subseteq \tau$. This contradicts $\tau_\text{I}$-goodness(a). So $j \neq 1$. If $j = 2$, then we have, with respect to $X^-$, a 2-alternating path $f_2 \rightarrow 2$ in $T(\tau)$ and a 2-alternating path $f_2 \rightarrow 3 \rightarrow f_3 \rightarrow 1 \rightarrow f_1 \rightarrow 2$ in $T(\tau_\text{I})$. Thus $j \neq 2$. The same argument using the path $f_3 \rightarrow 1 \rightarrow f_1 \rightarrow 2$ in $T(\tau_\text{I})$ shows that $j \neq 3$. This proves (c).
\end{proof}

Suppose $\tau \in \ms S$. Let
\[
P = \{ 1 \times f_{j_1}, i_1 \times f_{j_1}, i_1 \times f_{j_2}, i_2 \times f_{j_2}, \dotsc, i_{k-1} \times f_{j_k}, 4 \times f_{j_k} \}
\]
be the path in $T(\tau)$ from 1 to 4. By the same arguments as in Case 1, we have $k > 1$, $i_1 = 3$, and $f_{j_k} \neq f_2$. Furthermore, suppose that $f_{j_1} \neq f_3$. Let $\tau'$ be the minimal element of $\ms T(X^- \cup \{1 \times f_{j_1}, 3 \times f_{j_1}\})$ with respect to $\preceq_4$. Since $\tau' \preceq_4 \tau$, by Propostition~\ref{monotone} we have $N_{\tau'}(4) \subseteq N_\tau(4)$, and hence $4 \times f_1 \notin \tau'$. In addition, $\{1 \times f_{j_1}, 3 \times f_{j_1}, 3 \times f_3\} \subseteq \tau'$ and $f_{j_1} \neq f_3$, so $1 \times f_3 \notin \tau'$. Thus, $\tau' \in \ms S$. Now, by Proposition~\ref{minimal}, we must have $2 \times f_j \in \tau'$ for some $j = 1$, 2, 3, or $j_1$. However, all of these cases contradict Proposition~\ref{restrictions2}(c) or $\tau_\text{I}$-goodness(b). Hence $j_1 = 3$.

It follows that $P$ must satisfy one of the following.
\begin{enumerate} \renewcommand{\labelenumi}{(\roman{enumi})}
\item $P = \{1 \times f_3, 3 \times f_3, 3 \times f_{j_2}, 4 \times f_{j_2}\}$, $j_2 \neq 2$
\item $P = \{1 \times f_3, 3 \times f_3, 3 \times f_{j_2}, 2 \times f_{j_2}, 2 \times f_{j_3}, 4 \times f_{j_3}\}$, $j_3 \neq 2$.
\end{enumerate}
Now, choose $\tau$ to be a maximal element element of $\ms S$ with respect to $\preceq_{24}$. We consider the subcases (i) and (ii).

\paragraph{Subcase 2.1: $P$ has form (i)}

Let $Y = \{Y^+,Y^-\}$ be the circuit
\begin{align*}
Y^- &= \{ 1 \times f_1, 4 \times f_{j_2}, 3 \times f_3 \} \\
Y^+&= \{ 4 \times f_1, 3 \times f_{j_2}, 1 \times f_3 \}
\end{align*}
Let $\rho = Y \minus \{4 \times f_1\}$. Then $\rho$ is a maximal simplex of $\ms T_Y^+$ and $\rho \cup \{4 \times f_2\} = X^- \cup P \subseteq \tau$. Let $\tau_\ast$ be the minimal element of $\ms T(X^- \cup P)$ with respect to $\preceq_4$. Since $P \subseteq \tau_\ast$, we have $\tau_\ast \in \ms S$. By Proposition~\ref{minimal}, we have $2 \times f_j \in \tau_\ast$ for some $j = 1$, 2, 3, or $j_2$. By Proposition~\ref{restrictions2}(c), we thus have $2 \times j_2 \in \tau_\ast$. It follows by Proposition~\ref{minimal} that $\tau_\ast$ is the minimal element of $\ms T(Y^-)^\ast$ with respect to $\preceq_4$.

Since $P \subseteq \tau_\ast$, by Lemma~\ref{prec12}, we have $\tau \sim_{24} \tau_\ast$. Hence $\tau_\ast$ is a maximal element of $\ms S$ with respect to $\preceq_{24}$, and $P \subseteq \tau_\ast$. We may thus redefine $\tau$ as $\tau = \tau_\ast$.

We first claim that $\ms T(Y^-) \subseteq \ms T(X^-)$. Suppose that $\tau' \in \ms T(Y^-)^\ast$. Clearly $1 \times f_1$, $3 \times f_3 \in \tau'$. Since $\tau$ is minimal in $\ms T(Y^-)^\ast$ with respect to $\preceq_4$, by Proposition~\ref{monotone} we have $N_\tau(4) \subseteq N_{\tau'}(4)$. Hence $4 \times f_2 \in \tau'$. So $\tau' \in \ms T(X^-)$.

We now claim that $\ms T$ has a flip supported on $(Y^+,Y^-)$. Suppose the contrary. Since $\rho \in \ms T$, by Proposition~\ref{flip}, there is some $\tau' \in \ms T(Y^-)^\ast$ such that $\abs{\tau' \cap Y^+} \le \abs{Y^+}-2 = 1$ and $4 \times f_1 \notin \tau'$. As just shown, $\tau' \in \ms T(X^-)$. Moreover, by Proposition~\ref{monotone}, we have $N_\tau(f_2) \supseteq N_{\tau'}(f_2)$, so $3 \times f_2 \notin \tau'$. Hence $\tau' \in \ms S$. Since $\abs{\tau' \cap Y^+} \le 1$, we have $P \not\subseteq \tau'$, and hence by Lemma~\ref{prec12}, $\tau \not\sim_{24} \tau'$. Thus, $\tau \prec_{24} \tau'$. This contradicts the maximality of $\tau$ in $\ms S$ with respect to $\preceq_{24}$. Hence $\ms T$ has a flip supported on $(Y^+,Y^-)$.

Let $\ms T'$ be the result of this flip. We have
\[
\ms T'(X^-)^\ast = \ms T(X^-)^\ast \minus \ms T(Y^-)^\ast \cup \Psi(\ms T(X \minus \{3 \times f_{j_2}\})^\ast).
\]
Hence $\abs{\ms T'(X^-)^\ast} < \abs{\ms T(X^-)^\ast}$. By Proposition~\ref{Iflip1}, since $\rho \cup \{2 \times f_{j_2}\}$, we have $\abs{(\ms T')^\text{I}} \le \abs{\ms T^\text{I}}$. It is easy to check from the above equation that if $\ms T(X^-)$ is $\tau_\text{I}$-good, then so is $\ms T'(X^-)$. Finally, if $\tau_\text{I} \in \ms T(Y^-)$, then as in the previous paragraph we would have $3 \times f_2 \notin \tau_\text{I}$, a contradiction. So $\tau_\text{I} \notin \ms T(Y^-)$, and hence $\tau_\text{I} \in \ms T'(X^-)$.

\paragraph{Subcase 2.2: $P$ has form (ii)}

Let $Y = \{Y^+,Y^-\}$ be the circuit
\begin{align*}
Y^- &= \{ 1 \times f_1, 4 \times f_{j_3}, 2 \times f_{j_2}, 3 \times f_3 \} \\
Y^+&= \{ 4 \times f_1, 2 \times f_{j_3}, 3 \times f_{j_2}, 1 \times f_3 \}.
\end{align*}
Let $\rho = Y \minus \{4 \times f_1\}$. Then $\rho$ is a maximal simplex of $\ms T_Y^+$ and $\rho \subseteq \tau$. By Proposition~\ref{minimal}, $\tau$ is the minimal element of $\ms T(Y^-)^\ast$ with respect to $\preceq_4$.

By the same arguments as in the previous subcase, we have $\ms T(Y^-) \subseteq \ms T(X^-)$ and $\ms T$ has a flip supported on $(Y^+,Y^-)$. If the result of this flip is $\ms T'$, then Proposition~\ref{Iflip3} implies $\abs{(\ms T')^\text{I}} = \abs{\ms T^\text{I}}$. Similar arguments to the previous subcase show that the rest of Claim~\ref{case1claim} holds. This concludes Case 2.

\subsubsection{Case 3: $N_1 = \{1,2,3\}$, $N_2 = \{3,4\}$}

Let $X = \{X^+,X^-\}$ be the circuit
\begin{align*}
X^- &= \{ 3 \times f_1, 4 \times f_2\} \\
X^+&= \{ 4 \times f_1, 3 \times f_2\}
\end{align*}
Let $\sigma = X \minus \{4 \times f_1\}$. Then $\sigma$ is a maximal simplex of $\ms T_X^+$, and $\sigma_\text{I} = \sigma \cup \{1 \times f_1, 2 \times f_1\} \subseteq \tau_\text{I}$. Thus, if $(X^+,X^-)$ supports a flip of $\ms T$, then this flip satisfies the conditions of Proposition~\ref{Iflip2}(b), and we would be done.

By Proposition~\ref{minimal}, $\tau_\text{I}$ is minimal in $\ms T(X^-)^\ast$ with respect to $\preceq_4$. We can use the same argument as in Case 1 to prove that $\ms T(X^-)$ is $\tau_\text{I}$-good.

Let
\begin{align*}
\ms T(X^-)_1 &= \{ \tau \in \ms T(X^-)^\ast : 2 \times f_1 \in \tau,\, 1 \times f_1 \notin \tau,\, X^+ \cap \tau = \emptyset \} \\
\ms T(X^-)_2 &= \{ \tau \in \ms T(X^-)^\ast : 1 \times f_1 \in \tau,\, 2 \times f_1 \notin \tau,\, X^+ \cap \tau = \emptyset \}
\end{align*}
We will prove the following.

\begin{prop} \label{12notempty}
Let $\ms T$ be any triangulation of $\Delta^{m-1} \times \Delta^{n-1}$ such that $\tau_\text{I} \in \ms T$ and $\ms T(X^-)$ is $\tau_\text{I}$-good. If $\ms T$ does not have a flip supported on $(X^+,X^-)$, then $\ms T(X^-)_1$ and $\ms T(X^-)_2$ are not both empty.
\end{prop}

\begin{proof}
We first note the following

\begin{prop} \label{restrictions3}
Let $\ms T$ be as above. Let $\tau \in \ms T(X^-)^\ast$. Suppose that $\tau \cap X^+ = \emptyset$. The following are true.
\begin{enumerate} \renewcommand{\labelenumi}{(\alph{enumi})}
\item $1 \times f_2$, $2 \times f_2 \notin \tau$
\item $\{1 \times f_1, 2 \times f_1 \} \not\subseteq \tau$
\end{enumerate}
\end{prop}

\begin{proof}
Since $\tau_\text{I}$ is minimal in $\ms T(X^-)^\ast$ with respect to $\preceq_4$, by Proposition~\ref{monotone} we have $N_{\tau_\text{I}}(f_2) \supseteq N_\tau(f_2)$. This proves (a).

Now suppose $\{1 \times f_1, 2 \times f_1 \} \subseteq \tau$. Since $\tau \cap X^+ = \emptyset$, it follows that $\tau \in \ms T^\text{I}$ and $\sigma_\text{I} \not\subseteq \tau$. This contradicts $\tau_\text{I}$-goodness(a). This proves (b).
\end{proof}

Now suppose $\ms T$ does not have a flip supported on $(X^+,X^-)$. By Proposition~\ref{flip}, there is some $\tau \in \ms T(X^-)^\ast$ with $\tau \cap X^+ = \emptyset$. Let
\[
P = \{ 3 \times f_{j_1}, i_1 \times f_{j_1}, i_1 \times f_{j_2}, i_2 \times f_{j_2}, \dotsc, i_{k-1} \times f_{j_k}, 4 \times f_{j_k} \}
\]
be the path from 3 to 4 in $T(\tau)$. If $k > 2$ and $\{i_1,i_2\} = \{1,2\}$, then we have $\{1,2\} \subseteq N_\tau(f_{j_2})$ and $j_2 \neq 1$, 2. This contradicts $\tau_\text{I}$-goodness(b). Hence, $P$ has one of the following forms:
\begin{align*}
P &= \{3 \times f_{j_1}, 4 \times f_{j_1}\}, \text{ or } \\
P &= \{3 \times f_{j_1}, i_1 \times f_{j_1}, i_1 \times f_{j_2}, 4 \times f_{j_2}\}.
\end{align*}
In either case, let $\tau'$ be the minimal element of $\ms T(X^- \cup P)^\ast$ with respect to $\preceq_4$. Since $P \subseteq \tau'$, we have $\tau' \cap X^+ = \emptyset$.

Suppose $P$ is of the first form. By Proposition~\ref{minimal}, we have $1 \times f_j$, $2 \times f_{j'} \in \tau'$ for some $j$, $j' = 1$, 2, or $j_1$. By Proposition~\ref{restrictions3}(a), we have $j$, $j' \neq 2$. If $j = j' = 1$, this contradicts Proposition~\ref{restrictions3}(b). If $j = j' = j_1$, this contradicts $\tau_\text{I}$-goodness(b). Hence, exactly one of $j$, $j'$ is 1 and the other is $j_1$. If $j' = 1$ and $j = j_1$, then $\tau' \in \ms T(X^-)_1$. Otherwise, $\tau' \in \ms T(X^-)_2$, as desired.

Now suppose $P$ is of the second form. First assume $i_1 = 1$. By Proposition~\ref{minimal}, we have $2 \times f_j \in \tau'$ for some $j = 1$, 2, $j_1$, or $j_2$. If $j \in \{j_1,j_2\}$ and $j \neq 1$, 2, then we have a contradiction by $\tau_\text{I}$-goodness(b). Also, Proposition~\ref{restrictions3}(a) implies $j \neq 2$. Hence $j = 1$. Since we proved $j \neq j_1$, we have also shown that $j_1 \neq 1$. (Also, since $1 \times f_{j_2} \in P$, we have $j_2 \neq 2$ by Proposition~\ref{restrictions3}(a).) Hence $\tau' \in \ms T(X^-)_1$. By an analogous argument, if $i_1 = 2$ then $\tau' \in \ms T(X^-)_2$.
\end{proof}

To finish this case, it suffices to prove the following.

\begin{claim} \label{case3claim}
Let $\ms T$ be any triangulation of $\Delta^{m-1} \times \Delta^{n-1}$ such that $\tau_\text{I} \in \ms T$ and $\ms T(X^-)$ is $\tau_\text{I}$-good. Suppose that $\ms T(X^-)_1$ is nonempty. Then $\ms T$ is flip-connected to a triangulation $\ms T'$ such that $\tau_\text{I} \in \ms T'$, $\ms T'(X^-)$ is $\tau_\text{I}$-good, $\abs{(\ms T')^\text{I}} \le \abs{\ms T^\text{I}}$, $\abs{\ms T'(X^-)_2} \le \abs{\ms T(X^-)_2}$, and $\abs{\ms T'(X^-)} < \abs{\ms T(X^-)}$.
\end{claim}

Using this claim, we can apply flips that our original triangulation $\ms T$ to eventually reach some $\ms T'$ with $\abs{(\ms T')^\text{I}} \le \abs{\ms T^\text{I}}$, $\abs{\ms T'(X^-)_2} \le \abs{\ms T(X^-)_2}$, and $\ms T'(X^-)_1 = \emptyset$. By symmetry, the claim also holds after switching the roles of 1 and 2. Thus, we can show $\ms T'$ is flip-connected to some $\ms T''$ with $\abs{(\ms T'')^\text{I}} \le \abs{\ms T^\text{I}}$ and $\ms T''(X^-)_1 = \ms T''(X^-)_2 = \emptyset$. Moreover, the claim implies $\tau_\text{I} \in \ms T''$ and $\ms T''(X^-)$ is $\tau_\text{I}$-good. By Proposition~\ref{12notempty}, it follows that $\ms T''$ has a flip supported on $(X^+,X^-)$, and by Proposition~\ref{Iflip2}(b), this finishes the Case.

We now prove Claim~\ref{case3claim}. Let $\tau_0$ be a maximal element of $\ms T(X^-)_1$ with respect to $\preceq_{14}$. Let $P$ be the path from 3 to 4 in $T(\tau_0)$. By the proof of Proposition~\ref{12notempty}, $P$ satisifies one of the following.\footnote{We have $i_1 = 1$ in case (ii) by the definition of $\ms T(X^-)_1$. Also, notice that we have switched $j_2$ and $j_1$ in case (ii) from the original notation; this will allow us to later talk about both cases simultaneously.}
\begin{enumerate} \renewcommand{\labelenumi}{(\roman{enumi})}
\item $P = \{3 \times f_{j_1}, 4 \times f_{j_1}\}$, $j_1 \neq 1$, 2
\item $P = \{3 \times f_{j_2}, 1 \times f_{j_2}, 1 \times f_{j_1}, 4 \times f_{j_1}\}$, $j_2 \neq 1$, $j_1 \neq 2$
\end{enumerate}
We will deal with these subcases simultaneously, but our proof will require one extra step compared to the other two Cases. Let $\sigma_0$ be the unique minimal subset of $\tau_0$ such that $T(\sigma_0)$ has a connected component containing $\{2,3,4\}$. Call a collection $\ms S \subseteq \ms T$ of simplices $\tau_0$-\emph{good} if there is no maximal simplex $\tau \in \ms S$ such that $\tau \in \ms T(X^-)_1$ and $\sigma_0 \not\subseteq \tau$.

Now, if we have case (i), let $Y = \{Y^+,Y^-\}$ be the circuit
\begin{align*}
Y^- &= \{ 3 \times f_1, 4 \times f_{j_1} \} \\
Y^+&= \{ 4 \times f_1, 3 \times f_{j_1} \}
\end{align*}
If we have case (ii), let $Y = \{Y^+,Y^-\}$ be the circuit
\begin{align*}
Y^- &= \{ 3 \times f_1, 4 \times f_{j_1}, 1 \times f_{j_2} \} \\
Y^+&= \{ 4 \times f_1, 1 \times f_{j_1}, 3 \times f_{j_2} \}
\end{align*}
In either case, let $\rho = Y \minus \{4 \times f_1\}$. Since $\tau_0 \in \ms T(X^-)_1$, we have $2 \times f_1 \in \tau_0$. Hence $\sigma_0 = \rho \cup \{2 \times f_1\}$, and $\sigma_0 \cup \{4 \times f_2\} \subseteq \tau_0$.

In case (ii), we have by Proposition~\ref{minimal} that $\tau_0$ is the minimal element of $\ms T(Y^-)$ with respect to $\preceq_4$. Suppose we have case (i). Let $\tau_\ast$ be the minimal element of $\ms T(\sigma_0 \cup \{4 \times f_2\})$ with respect to $\preceq_4$. By the proof of Proposition~\ref{12notempty}, since $2 \times f_1 \in \tau_\ast$, we have $1 \times f_{j_1} \in \ms \tau_\ast$ and $\tau_\ast \in \ms T(X^-)_1$. By Proposition~\ref{minimal}, it follows that $\tau_\ast$ is minimal in $\ms T(Y^-)$ with respect to $\preceq_4$. Also, since $\sigma_0 \subseteq \tau_\ast$, by Lemma~\ref{prec12} we have $\tau_0 \sim_{14} \tau_\ast$. Hence $\tau_\ast$ is a maximal element of $\ms T(X^-)_1$ with respect to $\preceq_{14}$, and $P \subseteq \tau_\ast$. We may thus redefine $\tau_0$ as $\tau_0 = \tau_\ast$.

In either case, since $\tau_0$ is minimal in $\ms T(Y^-)$ with respect to $\preceq_4$, the same argument as in the previous Cases shows that $\ms T(Y^-)$ is $\tau_0$-good.

We now reach our final claim.

\begin{claim} \label{case3subclaim1}
Let $\ms T$ be any triangulation of $\Delta^{m-1} \times \Delta^{n-1}$ such that $\tau_\text{I}$, $\tau_0 \in \ms T$, $\ms T(X^-)$ is $\tau_\text{I}$-good, and $\ms T(Y^-)$ is $\tau_0$-good. Then either there is a flip of $\ms T$ supported on $(Y^+,Y^-)$, or there is a flip of $\ms T$ with result $\ms T'$ such that
\begin{multicols}{2}
\begin{enumerate} \renewcommand{\labelenumi}{(\alph{enumi})}
\item $\tau_\text{I}$, $\tau_0 \in \ms T'$
\item $\ms T'(X^-)$ is $\tau_\text{I}$-good
\item $\ms T'(Y^-)$ is $\tau_0$-good
\item $\abs{(\ms T')^\text{I}} \le \abs{\ms T^\text{I}}$
\item $\abs{\ms T'(X^-)_2} \le \abs{\ms T(X^-)_2}$
\item $\abs{\ms T'(X^-)^\ast} \le \abs{\ms T(X^-)^\ast}$
\item $\abs{\ms T'(Y^-)^\ast} < \abs{\ms T(Y^-)^\ast}$
\end{enumerate}
\end{multicols}
\end{claim}

Let us see how this will complete the proof of Case 3. Suppose $\ms T$ satisfies (a)-(c). We first show that $\ms T(Y^-) \subseteq \ms T(X^-)$ and $\tau_\text{I} \notin \ms T'(Y^-)$. For the first claim, let $\tau \in \ms T(Y^-)$. Clearly $3 \times f_1 \in \tau$. Also, by Proposition~\ref{minimal}, $\tau_0$\ is minimal in $\ms T(Y^-)$ with respect to $\preceq_4$. Hence by Proposition~\ref{monotone}, we have $N_{\tau_0}(4) \subseteq N_\tau(4)$, so $4 \times f_2 \in \tau$. Thus $\tau \in \ms T(X^-)$, as desired. For the second claim, if $\tau_\text{I} \in \ms T(Y^-)$, then by Proposition~\ref{monotone} $N_{\tau_0}(f_2) \supseteq N_{\tau_\text{I}}(f_2)$, so $3 \times f_2 \notin \tau_\text{I}$, a contradiction. Hence $\tau_\text{I} \notin \ms T(Y^-)$.

Now, by Claim~\ref{case3subclaim1}, $\ms T$ is flip-connected to a triangulation $\ms T'$ that satisfies (a)-(f) of the Claim and such that there is a flip of $\ms T'$ supported on $(Y^+,Y^-)$. Let $\ms T''$ be the result of this flip. In case (i), we have
\[
\ms T''(X^-)^\ast = \ms T'(X^-)^\ast \minus \ms T'(Y^-)^\ast \cup \Psi( \ms T'(Y \minus \{3 \times f_{j_1}\})^\ast )
\]
and in case (ii), we have
\[
\ms T''(X^-)^\ast = \ms T'(X^-)^\ast \minus \ms T'(Y^-)^\ast \cup \Psi( \ms T'(Y \minus \{1 \times f_{j_1}\})^\ast ) \cup \Psi( \ms T'(Y \minus \{3 \times f_{j_2}\})^\ast ).
\]
In both cases, $\abs{\ms T''(X^-)^\ast} < \abs{\ms T'(X^-)^\ast}$. We now check the rest of the conclusions of Claim~\ref{case3claim}. By Proposition~\ref{Iflip2}(a) in case (i) and Proposition~\ref{Iflip1}(a) in case (ii), we have $\abs{(\ms T'')^\text{I}} \le \abs{(\ms T')^\text{I}}$. It is easy to check from the above equations that if $\ms T'(X^-)$ is $\tau_\text{I}$-good, then so is $\ms T''(X^-)$. As shown above, $\tau_\text{I} \notin \ms T'(Y^-)$, so $\tau_\text{I} \in \ms T''$. Finally, it is easy to check that if $\Psi(\tau) \in \ms T''(X^-)_2$, then $\tau \in \ms T'(X^-)_2$. Hence, $\abs{\ms T''(X^-)_2} \le \abs{\ms T'(X^-)_2}$. This proves Claim~\ref{case3claim}.

We now prove Claim~\ref{case3subclaim1}.

\begin{proof}[Proof of Claim~\ref{case3subclaim1}]
Suppose $\ms T$ does not have a flip supported on $(Y^+,Y^-)$. Let $\ms S$ be the set of maximal simplices $\tau \in \ms T(Y^-)^\ast$ such that $\abs{\tau \cap Y^+} \le \abs{Y^+} - 2$ and $4 \times f_1 \notin \ms \tau$. Since $\rho \in \ms T$, by Proposition~\ref{flip}, $\ms S \neq \emptyset$.

As shown above, $\tau_0$ is the minimal element of $\ms T(Y^-)^\ast$ with respect to $\preceq_4$ and $\ms T(Y^-) \subseteq \ms T(X^-)$. We also prove the following.

\begin{prop} \label{restrictions4}
Let $\tau \in \ms S$. The following are true.
\begin{enumerate} \renewcommand{\labelenumi}{(\alph{enumi})}
\item $\tau \cap X^+ = \emptyset$
\item $2 \times f_{j_r} \notin \tau$ for all $r$
\item $i \times f_j \notin \tau$ for all $i$, $j \in \{1, 2\}$
\end{enumerate}
\end{prop}

\begin{proof}
By definition of $\ms S$, we have $4 \times f_1 \notin \tau$. Since $\tau_0$\ is minimal in $\ms T(Y^-)$ with respect to $\preceq_4$, by Proposition~\ref{monotone} we have $N_{\tau_0}(f_2) \supseteq N_{\tau}(f_2)$, so $3 \times f_2 \notin \tau_\text{I}$. This proves (a).

Suppose $2 \times f_{j_r} \in \tau$. Then with respect to $Y^-$, we have a 2-alternating path $f_{j_r} \rightarrow 2$ in $T(\tau)$, and a 2-alternating path $f_{j_r} \rightarrow \dotsc \rightarrow 3 \rightarrow f_1 \rightarrow 2$ in $T(\tau_0)$. This contradicts Proposition~\ref{alternating}, which proves (b).

By part (a) and Proposition~\ref{restrictions3}(a), we have $i \times f_2 \notin \tau$ for $i = 1$, 2. Suppose that $1 \times f_1 \in \tau$. Let $\tau'$ be the minimal element of $\ms T(X^- \cup Y^- \cup \{1 \times f_1\})$ with respect to $\preceq_4$. By Proposition~\ref{minimal}, we have $2 \times f_j$ for some $j = 1$, 2, or $j_r$. This contradicts Propositions~\ref{restrictions3}(b), (a), and \ref{restrictions4}(b) respectively. Thus $1 \times f_1 \notin \tau$.

Finally, suppose that $2 \times f_1 \in \tau$. As above, we have $1 \times f_1 \notin \tau$. Hence $\tau \in \ms T(X^-)_1$. But since $\abs{\tau \cap Y^+} \le \abs{Y^+} - 2$, we have $\sigma_0 \not\subseteq \tau$, contradicting $\tau_0$-goodness. This proves (c).
\end{proof}

Now, let
\[
\ms S' = \{ \tau \in \ms S : 1 \times f_{j_1} \in \tau \}.
\]

\begin{prop} \label{s'notempty}
$\ms S'$ is not empty.
\end{prop}

\begin{proof}
Suppose $\tau \in \ms S$. Let $P'$ be the path in $T(\tau)$ from 3 to 4. By Proposition~\ref{restrictions4}(a) and the proof of Proposition~\ref{12notempty}, $P'$ has one of the following forms.
\begin{enumerate} \renewcommand{\labelenumi}{(\roman{enumi}$'$)}
\item $P' = \{3 \times f_{j_1'}, 4 \times f_{j_1'}\}$, $j_1' \neq 1$, 2
\item $P' = \{3 \times f_{j_1'}, 2 \times f_{j_1'}, 2 \times f_{j_2'}, 4 \times f_{j_2'}\}$, $j_1' \neq 1$, $j_2' \neq 2$
\item $P' = \{3 \times f_{j_1'}, 1 \times f_{j_1'}, 1 \times f_{j_2'}, 4 \times f_{j_2}\}$, $j_1' \neq 1$, $j_2' \neq 2$
\end{enumerate}
Suppose we have (iii$'$). Then by the proof of Proposition~\ref{12notempty}, there is some $\tau' \in \ms T(X^-)_1$ with $P' \subseteq \tau'$. Since $P' \neq P$, we have $\sigma_0 \not\subseteq \tau'$. This contradicts $\tau_0$-goodness. Hence (i$'$) or (ii$''$) must hold.

Now, let $\tau'$ be the minimal element of $\ms T(X^- \cup Y^- \cup P')$ with respect to $\preceq_4$. Since $P' \subseteq \tau'$, we have $\tau' \in \ms S$. By Proposition~\ref{minimal}, we have $1 \times f_j \in \tau'$ for $j = 1$, 2, $j_1$, or $j_r'$ for some $r$.\footnote{This is true in both cases (i) and (ii). In case (ii), consider the 2-alternating path from 4 to 1. The second to last edge of this path must be $i \times f_j \in X^- \cup Y^- \cup P'$ for some $i \neq 1$; therefore $j = 1$, 2, $j_1$, or $j_r'$.} By Proposition~\ref{restrictions4}(c), we have $j \neq 1$, 2.

Suppose $j = j_r'$. First assume (i$'$) is true. Then $1 \times f_{j_1'} \in \tau'$. On the other hand, we also have $2 \times f_{j'} \in \tau'$ for some $j = 1$, 2, $j_r$ or $j_1'$, and so by Propositions~\ref{restrictions4}(b) and (c), we have $2 \times f_{j_1'} \in \tau'$. This contradicts $\tau_\text{I}$-goodness(b). If (ii$'$) is true, then we immediately have a contradiction to $\tau_\text{I}$-goodness(b). Hence $j = j_1$. So $\tau' \in \ms S'$, as desired.

Note that in this proof, we showed that $1 \times f_{j_r'} \notin \tau'$ for any $r$. Since $1 \times f_{j_1} \in \tau'$, and in case (ii), $1 \times f_{j_2} \in \tau'$ since $Y^- \subseteq \tau'$, it follows that $f_{j_r'} \neq j_1$, $j_2$ for all $r$.
\end{proof}

Now, choose $\tau$ to be a maximal element of $\ms S'$ with respect to $\preceq_{23}$. Let $P'$ be the path in $T(\tau)$ from 3 to 4. Then as in the previous proof, either (i$'$) or (ii$'$) is true. We consider these cases separately.

\paragraph{Subcase 3.1: $P'$ has form (i$'$)}

As noted in the previous proof, we have $j_1' \neq j_1$, $j_2$. In case (i), let $Z = \{Z^+,Z^-\}$ be the circuit
\begin{align*}
Z^- &= \{ 3 \times f_{j_1'}, 4 \times f_{j_1} \} \\
Z^+&= \{ 4 \times f_{j_1'}, 3 \times f_{j_1} \}
\end{align*}
and in case (ii), let $Z$ be the circuit
\begin{align*}
Z^- &= \{ 3 \times f_{j_1'}, 4 \times f_{j_1}, 1 \times f_{j_2} \} \\
Z^+&= \{ 4 \times f_{j_1'}, 1 \times f_{j_1}, 3 \times f_{j_2} \}
\end{align*}
Let $\pi = Z \minus \{ 3 \times f_{j_t} \}$, where $t = 1$ in case (i) and $t = 2$ in case (ii). Then $\pi$ is a maximal simplex of $\ms T_Z^+$. Since $1 \times f_{j_1} \in \tau$ by the definition of $\ms S'$, we have $\pi \cup X^- \cup \{1 \times f_{j_1}\} = X^- \cup Y^- \cup P' \cup \{1 \times f_{j_1}\} \subseteq \tau$.
Let $\tau_\ast$ be the minimal element of $\ms T(X^- \cup Y^- \cup P' \cup \{1 \times f_{j_1}\})$ with respect to $\preceq_3$. By the proof of Proposition~\ref{s'notempty}, we have $2 \times f_{j_1'} \in \tau_\ast$. So by Proposition~\ref{minimal}, $\tau_\ast$ is minimal in $\ms T(Z^-)$ with respect to $\preceq_3$.

Since $P' \cup \{1 \times f_{j_1}, 4 \times f_{j_1}\} \subseteq \tau$, $\tau^\ast$, by Lemma~\ref{prec12} we have $\tau \sim_{23} \tau_\ast$. Hence $\tau_\ast$ is maximal in $\ms S'$ with respect to $\preceq_{23}$. We can thus redefine $\tau$ as $\tau = \tau^\ast$.

We claim that $\ms T(Z^-) \subseteq \ms T(Y^-)$. Suppose $\tau' \in \ms T(Z^-)$. We have $Y^- \minus \{3 \times f_1\} \subseteq Z^- \subseteq \tau'$. Also, since $\tau$ is minimal in $\ms T(Z^-)$ with respect to $\preceq_3$, we have $N_\tau(3) \subseteq N_{\tau'}(3)$. Hence $3 \times f_1 \in \tau'$. So $\tau' \in \ms T(Y^-)$.

We next claim that $\pi' := Z \minus \{4 \times f_{j_1'}\} \in \ms T$. By Proposition~\ref{elimination}, there exists a maximal simplex $\tau' \in \ms T(Z^-)$ with $\tau' = \tau \minus \{4 \times f_{j_1'}\} \cup \{i \times f_j\}$ for some $i \times f_j$ where $i \in \{2,3\}$. Then $\tau \prec_3 \tau'$, so by the definition of $\tau$, we must have $\tau' \notin \ms S'$. Hence, we must have $i \times f_j = 3 \times f_{j_t}$. Thus $\pi' \subseteq \tau'$, as desired.

We finally claim that $\ms T$ has a flip supported on $(Z^+,Z^-)$. Suppose the contrary. Since $\pi' \in \ms T$, by Proposition~\ref{flip}, there is some $\tau' \in \ms T(Z^-)$ with $\abs{\tau' \cap Z^+} \le \abs{Z^+} - 2$ and $4 \times f_{j_1'} \notin \tau'$. As shown above, $\tau' \in \ms T(Y^-)$. In addition, since $\tau$ is minimal in $\ms T(Z^-)$ with respect to $\preceq_3$, we have $N_\tau(f_1) \supseteq N_{\tau'}(f_1)$, so $4 \times f_1 \notin \tau'$. It follows that $\abs{\tau' \cap Y^+} \le \abs{\tau' \cap Z^+} \le \abs{Z^+} - 2 = \abs{Y^+}-2$. Hence, $\tau' \in \ms S$.

Let $P''$ be the path from 3 to 4 in $T(\tau')$, and let $\tau''$ be the minimal element of $\ms T(X^- \cup Y^- \cup P'')$ with respect to $\preceq_4$. By the proof of Proposition~\ref{s'notempty}, $\tau'' \in \ms S'$. Since $4 \times f_{j_1'} \notin P''$ by the definition of $\tau'$, we have $P' \not\subseteq \tau''$. Thus, by Lemma~\ref{prec12}, we have $\tau \not\sim_{23} \tau''$. It follows that $\tau \prec_{23} \tau''$, which contradicts the maximality of $\tau$ in $\ms S'$ with respect to $\preceq_{23}$.

Hence, $\ms T$ has a flip supported on $(Z^+,Z^-)$. Let $\ms T'$ be the result of this flip. In both cases, we have
\[
\ms T'(Y^-)^\ast = \ms T(Y^-)^\ast \minus \ms T(Z^-)^\ast \cup \Psi( \ms T(Z \minus \{4 \times f_{j_1'}\})^\ast ).
\]
Hence $\abs{\ms T'(Y^-)^\ast} < \abs{\ms T(Y^-)^\ast}$. It is also clear from $\Psi$ that $\abs{\ms T'(X^-)^\ast} = \abs{\ms T(X^-)^\ast}$. By Proposition~\ref{Iflip2}(a) in case (i) and Proposition~\ref{Iflip1}(a) in case (ii), $\abs{(\ms T')^\text{I}} \le \abs{\ms T^\text{I}}$. It is easy to check that if $\Psi(\tau) \in \ms T''(X^-)_2$, then $\tau \in \ms T'(X^-)_2$. It is also easy to check $\tau_\text{I}$-goodness and $\tau_0$-goodness.

Finally, to see that $\tau_\text{I}$, $\tau_0 \in \ms T'$, it suffices to show $\tau_\text{I}$, $\tau_0 \notin \ms T(Z^-)$. Since $\tau$ is minimal in $\ms T(Z^-)$ with respect to $\preceq_3$, for any $\tau' \in \ms T(Z^-)$ we have $N_\tau(f_1) \supseteq N_{\tau'}(f_1)$ and hence $2 \times f_1 \notin \tau'$. Thus $\tau_\text{I}$, $\tau_0 \notin \ms T(Z^-)$.

\paragraph{Subcase 3.2: $P'$ has form (ii$'$)}

In case (i), let $Z = \{Z^+,Z^-\}$ be the circuit
\begin{align*}
Z^- &= \{ 3 \times f_{j_1'}, 2 \times f_{j_2'}, 4 \times f_{j_1} \} \\
Z^+&= \{ 2 \times f_{j_1'}, 4 \times f_{j_2'}, 3 \times f_{j_1} \}
\end{align*}
and in case (ii),
\begin{align*}
Z^- &= \{ 3 \times f_{j_1'}, 2 \times f_{j_2'}, 4 \times f_{j_1}, 1 \times f_{j_2} \} \\
Z^+&= \{ 2 \times f_{j_1'}, 4 \times f_{j_2'}, 1 \times f_{j_1}, 3 \times f_{j_2} \}
\end{align*}
Let $\pi = Z \minus \{ 3 \times f_{j_t} \}$, where $t = 1$ in case (i) and $t = 2$ in case (ii). Then $\pi$ is a maximal simplex of $\ms T_Z^+$ and $\pi \cup X^- \cup \{1 \times f_{j_1}\} \subseteq \tau$. By Proposition~\ref{minimal}, $\pi$ is minimal in $\ms T(X^-)$ with respect to $\preceq_3$.

By the same arguments as in the previous subcase, we have $\ms T(Z^-) \subseteq \ms T(Y^-)$, $Z \minus \{4 \times f_{j_2'}\} \in \ms T$, and $\ms T$ has a flip supported on $(Z^+,Z^-)$. The remainder of Claim~\ref{case3subclaim1} also follows the arguments of the previous subcase. This concludes the proof of Phase I.
\end{proof}

\subsection{Phase II}

Let $\ms T$ be a triangulation of $\Delta^3 \times \Delta^{n-1}$. Let $\ms T^\text{II}$ be the set of all maximal simplices $\tau \in \ms T^\ast$ for which there is some $f_j \in \Delta^{n-1}$ with $\{1,2\} \subseteq N_\tau(f_j)$ and $\{3,4\} \not\subseteq N_\tau(f_j)$. The goal of this section is to prove the following.

\begin{claim} \label{phaseIIclaim}
$\ms T$ is flip-connected to a triangulation $\ms T'$ with $(\ms T')^\text{II} = \emptyset$.
\end{claim}

By Phase I, we may assume $\ms T^\text{I} = \emptyset$. Assume $\ms T^\text{II} \neq \emptyset$.

\begin{prop}
Let $\tau \in \ms T^\text{II}$. Then $\sh(\tau) = \{\{1,2,4\},\{1,3\}\}$ or $\{\{1,2,4\},\{2,3\}\}$.
\end{prop}

\begin{proof}
Since $\ms T^\text{I} = \emptyset$, we have $\{1,2,4\} \subseteq \sh(\tau)$. Then either the Proposition is true or $\sh(\tau) = \{\{1,2,4\},\{3,4\}\}$. Suppose we have the latter case. Let $f_j \in \Delta^{n-1}$ be such that $N_\tau(f_j) = \{1,2,4\}$. By Proposition~\ref{elimination}, there exists a maximal simplex $\tau' \in \ms T^\ast$ with $\tau' = \tau \minus \{4 \times f_j\} \cup \{i' \times f_{j'}\}$ for some $j' \neq j$. Then $\tau' \in \ms T^\text{I}$, contradicting $\ms T^\text{I} = \emptyset$.
\end{proof}

Now, let $\tau_\text{II}$ be a maximal element of $\ms T^\text{II}$ with respect to $\preceq_3$. Without loss of generality, assume $\sh(\tau_\text{II}) = \{\{1,2,4\},\{1,3\}\}$; the case where $\sh(\tau_\text{II}) = \{\{1,2,4\},\{2,3\}\}$ follows analogously. Without loss of generality, assume $N_{\tau_\text{II}}(f_1) = \{1,2,4\}$ and $N_{\tau_\text{II}}(f_2) = \{1,3\}$.

Let $X = \{X^+,X^-\}$ be the circuit
\begin{align*}
X^- &= \{ 1 \times f_1, 3 \times f_2 \} \\
X^+&= \{ 3 \times f_1, 1 \times f_2 \}
\end{align*}
Let $\sigma = X \minus \{3 \times f_1\}$. Then $\sigma \cup \{2 \times f_1, 4 \times f_1\} \subseteq \tau_\text{II}$. By Proposition~\ref{minimal}, $\tau_\text{II}$ is the minimal element of $\ms T(X^-)^\ast$ with respect to $\preceq_3$.

We now prove the following.

\begin{prop} \label{goodII}
The following are true.
\begin{enumerate} \renewcommand{\labelenumi}{(\alph{enumi})}
\item There is no $\tau \in \ms T(X^-)^\ast$ with $\tau \in \ms T^\text{II}$ and $\tau \neq \tau_\text{II}$.
\item There is no $\tau \in \ms T(X^-)^\ast$ and $j \neq 1$, 2, such that $\{1,2\} \subseteq N_\tau(f_j)$.
\end{enumerate}
\end{prop}

\begin{proof}
To prove (a), assume such a $\tau$ exists. Since $\tau_\text{II}$ is minimal in $\ms T(X^-)^\ast$ with respect to $\preceq_3$, we have $\tau_\text{II} \prec_3 \tau$. This contradicts the maximality of $\tau_\text{II}$ in $\ms T^\text{II}$.

To prove (b), assume such a $\tau$ exists. Since $j \neq 2$, by Proposition~\ref{elimination}, there is some $\tau' \in \ms T(X^-)^\ast$ (possibly the same as $\tau$) with $N_{\tau'}(f_j) = N_\tau(f_j) \minus \{3\}$. Hence $\tau' \in \ms T^\text{II}$. Also, since $j \neq 1$, $\tau' \neq \tau_\text{II}$. This contradicts (a), which completes the proof.
\end{proof}

Call any subcollection $\ms S \subseteq \ms T$ of simplices $\tau_\text{II}$-\emph{good} if it satisfies Proposition~\ref{goodII} with $\ms T(X^-)$ replaced with $\ms S$. To prove Claim~\ref{phaseIIclaim}, it suffices to prove the following.

\begin{claim}
Let $\ms T$ be any triangulation of $\Delta^3 \times \Delta^{n-1}$ such that $\tau_\text{II} \in \ms T$ and $\ms T(X^-)$ is $\tau_\text{II}$-good. Then either there is a flip of $\ms T$ supported on $(X^+,X^-)$, or there is a flip of $\ms T$ with result $\ms T'$ such that $\tau_\text{II} \in \ms T'$, $\ms T'(X^-)$ is $\tau_\text{II}$-good, $\abs{(\ms T')^\text{II}} \le \abs{\ms T^\text{II}}$, and $\abs{\ms T'(X^-)^\ast} < \abs{\ms T(X^-)^\ast}$.
\end{claim}

The proof of this claim is the same as the proof of Claim~\ref{case1claim} in Phase I, Case 1, with the roles of 4 and 3 reversed. We leave the details to the reader.

\subsection{Phase III}

We will now define a specific triangulation $\ms T_0$ and show that all triangulations of $\Delta^3 \times \Delta^{n-1}$ are flip-connected to it. This will conclude the proof of the Theorem.

Let $\le_0$ be the total ordering on $\Delta^{n-1}$ such that $f_i <_0 f_j$ if $i < j$. We will show the following.

\begin{prop} \label{basepoint}
There is a unique triangulation $\ms T_0$ of $\Delta^3 \times \Delta^{n-1}$ such that $(\ms T_0)^\text{II} = \emptyset$ and the orders $\le_{\ms T_0[12]}$ and $\le_{\ms T_0[34]}$ are both the same as $\le_0$.
\end{prop}

\begin{proof}
We first prove the following.

\begin{prop} \label{orders}
Let $\ms T$ be a triangulation of $\Delta^3 \times \Delta^{n-1}$. Then $\ms T^\text{II} = \emptyset$ if and only if the orders $\le_{12}$, $\le_{13}$, $\le_{14}$, $\le_{32}$, $\le_{42}$ are all the same.
\end{prop}

\begin{proof}
Suppose $\ms T^\text{II} = \emptyset$. It suffices to prove that $\le_{12}$ and $\le_{13}$ are the same; the other cases are analogous. Suppose without loss of generality that $f_1 <_{12} f_2$ and $f_2 <_{13} f_1$. By Proposition~\ref{compare}, we have $\xi := \{1 \times f_1, 3 \times f_2 \} \in \ms T$. Let $\tau$ be the minimal element of $\ms T(\xi)^\ast$ with respect to $\preceq_3$. By Proposition~\ref{minimal}, we must have $2 \times f_j \in \tau$ for $j = 1$ or 2. If $j = 2$, then Proposition~\ref{compare} implies that $f_2 <_{12} f_1$, which contradicts our original assumption. Hence $2 \times f_1 \in \tau$. Now, since $\ms T^\text{II} = \emptyset$, this implies that $\{1,2,3,4\} \subseteq N_\tau(f_1)$. However, then there is no 1-alternating path with respect to $\xi$ in $T(\tau)$ from 3 to 2 or 4. This contradicts Propostion~\ref{minimal}. Hence $\le_{12}$ and $\le_{13}$ are the same.

Conversely, suppose $\ms T^\text{II} \neq \emptyset$. Let $\tau \in \ms T^\ast$ and $j \in [n]$ be such that $\{1,2\} \subseteq N_\tau(f_j)$ but $i \notin N_\tau(f_j)$ for either $i = 3$ or 4. Let $j' \in [n]$ be such that $i \times f_{j'} \in \tau$. By Proposition~\ref{compare}, we have $f_{j'} <_{1i} f_j$ and $f_j <_{i2} f_{j'}$. Thus $\le_{1i}$ and $\le_{i2}$ are not the same relation, as desired. 
\end{proof}

We now consider a particular drawing of $G_{m,n}$ in the plane. Let $\ell_1$ and $\ell_2$ be two parallel lines in the plane. Draw vertices 1, 3, 4, and 2 on $\ell_1$ in that order. Draw vertices $f_1$, $f_2$, \dots, $f_n$ on $\ell_2$ in that order and in the opposite direction as $\overrightarrow{12}$. This gives a drawing $D$ of $G_{m,n}$.

Suppose $\ms T$ is a triangulation of $\Delta^3 \times \Delta^{n-1}$ such that $\ms T^\text{II} = \emptyset$ and $\le_{\ms T[12]}$, $\le_{\ms[34]}$ are the same as $\le_0$. We claim that for any $\sigma \in \ms T$, the drawing of $T(\sigma)$ in $D$ is planar. Suppose the contrary. Then there exists $\{i \times f_j, i' \times f_{j'}\} \in \ms T$ such that $i$, $i'$ are distinct and appear in the sequence 1, 3, 4, 2 in that order, and $j < j'$. By Proposition~\ref{compare}, this implies $f_{j'} <_{ii'} f_j$. However, by our original assumptions on $\ms T$ and Proposition~\ref{orders}, we have that $\le_{ii'}$ is the same as $\le_0$. Hence $f_{j'} <_0 f_j$, so $j' < j$. This is a contradiction, which proves our claim.

Let $\ms T_0$ be the collection of all simplices $\sigma \subseteq \Delta^3 \times \Delta^{n-1}$ such that $T(\sigma)$ is planar in $D$. It is well-known that $\ms T_0$ is a triangulation of $\Delta^3 \times \Delta^{n-1}$. By the above claim, we have $\ms T \subseteq \ms T_0$, and hence $\ms T = \ms T_0$, as desired.
\end{proof}

We can now state the goal of Phase III.

\begin{claim}
If $\ms T$ is a triangulation of $\Delta^3 \times \Delta^{n-1}$, then $\ms T$ is flip-connected to $\ms T_0$.
\end{claim}

By Phase II, we may assume $\ms T^\text{II} = \emptyset$. By Proposition~\ref{basepoint}, it suffices to prove the following.

\begin{claim}
Let $\ms T$ be a triangulation with $\ms T^\text{II} = \emptyset$. Then the following are true.
\begin{enumerate} \renewcommand{\labelenumi}{(\alph{enumi})}
\item Suppose $f_j$ and $f_{j'}$ are consecutive increasing elements in the order $\le_{34}$. Then there is a flip of $\ms T$ with result $\ms T'$ such that $(\ms T')^\text{II} = \emptyset$, $\le_{\ms T[12]}$ and $\le_{\ms T'[12]}$ are the same, and $\le_{\ms T[34]}$ and $\le_{\ms T'[34]}$ are the same except $f_{j'} <_{\ms T'[34]} f_j$.
\item Suppose $f_j$ and $f_{j'}$ are consecutive increasing elements in the order $\le_{12}$. Then $\ms T$ is flip-connected to a triangulation $\ms T'$ such that $(\ms T')^\text{II} = \emptyset$, $\le_{\ms T[34]}$ and $\le_{\ms T'[34]}$ are the same, and $\le_{\ms T[12]}$ and $\le_{\ms T'[12]}$ are the same except $f_{j'} <_{\ms T'[12]} f_j$.
\end{enumerate}
\end{claim}

\subsubsection{Proof of (a)}

Let $X = \{X^+,X^-\}$ be the circuit
\begin{align*}
X^- &= \{ 4 \times f_j, 3 \times f_{j'} \} \\
X^+&= \{ 3 \times f_j, 4 \times f_{j'} \}
\end{align*}
By Proposition~\ref{compare}, $X^- \in \ms T$. We claim that $\ms T$ has a flip supported on $(X^+,X^-)$. If so, then by Proposition~\ref{fliporder}, we have that $\le_{\ms T[12]}$ and $\le_{\ms T'[12]}$ are the same and $\le_{\ms T[34]}$ and $\le_{\ms T'[34]}$ are the same except $f_{j'} <_{\ms T'[34]} f_j$. By Proposition~\ref{orders}, we also have $\ms T^\text{II} = \emptyset$. This will prove (a).

Suppose that $\ms T$ does not have a flip supported on $(X^+,X^-)$. By Proposition~\ref{flip}, there is a maximal simplex $\tau \in \ms T(X^-)^\ast$ with $\tau \cap X^+ = \emptyset$. Let
\[
P = \{ 4 \times f_{j_1}, i_1 \times f_{j_1}, i_1 \times f_{j_2}, i_2 \times f_{j_2}, \dotsc, i_{k-1} \times f_{j_k}, 3 \times f_{j_k} \}
\]
be the path in $T(\tau)$ from 4 to 3. First suppose that
\[
P = \{ 4 \times f_{j_1}, 3 \times f_{j_1} \}.
\]
Since $\ms T \cap X^+ = \emptyset$, we have $j_1 \neq j$, $j'$. Since $X^- \in \tau$, it follows from Proposition~\ref{compare} that $f_j <_{34} f_{j_1}$ and $f_{j_1} <_{34} f_{j'}$. This contradicts the assumption that $f_j$, $f_{j'}$ are consecutive in $\le_{34}$. Hence, $k > 1$.

Assume that $i_r = 1$ for some $r = 1$, \dots, $k-1$. Since $\{i_{r-1} \times f_{j_r}, 1 \times f_{j_{r+1}}\} \subseteq P$,\footnote{We define $i_0 = 4$, and $i_k = 3$.} we have $f_{j_r} <_{1i_{r-1}} f_{j_{r+1}}$. Since $\{1 \times f_{j_r}, i_{r+1} \times f_{j_{r+1}}\} \subseteq P$, we have $f_{j_{r+1}} <_{1i_{r+1}} f_{j_r}$. However, by Proposition~\ref{orders}, $<_{1i_{r-1}}$ and $<_{1i_{r+1}}$ are the same relation. This is a contradiction. Similarly, we cannot have $i_r = 2$ for any $r = 1$, \dots, $k-1$. Hence $\ms T$ has a flip supported on $(X^+,X^-)$, as desired.

\subsection{Proof of (b)}

By part (a), we may assume that $\le_{34}$ is the same as $\le_{12}$. Let $X_1$, $X_2$, \dots, $X_5$ be the circuits
\begin{gather*}
\begin{aligned}
X_1^- &= \{ 3 \times f_j, 1 \times f_{j'} \} \\
X_1^+&= \{ 1 \times f_j, 3 \times f_{j'} \}
\end{aligned} \qquad
\begin{aligned}
X_2^- &= \{ 2 \times f_j, 4 \times f_{j'} \} \\
X_2^+&= \{ 4 \times f_j, 2 \times f_{j'} \}
\end{aligned}\qquad
\begin{aligned}
X_3^- &= \{ 2 \times f_j, 1 \times f_{j'} \} \\
X_3^+&= \{ 1 \times f_j, 2 \times f_{j'} \}
\end{aligned} \\
\begin{aligned}
X_4^- &= \{ 2 \times f_j, 3 \times f_{j'} \} \\
X_4^+&= \{ 3 \times f_j, 2 \times f_{j'} \}
\end{aligned} \qquad
\begin{aligned}
X_5^- &= \{ 4 \times f_j, 1 \times f_{j'} \} \\
X_5^+&= \{ 1 \times f_j, 4 \times f_{j'} \}  
\end{aligned}
\end{gather*}
We claim that $\ms T$ has a flip supported on $(X_1^+, X_1^-)$, the result has a flip supported on $(X_2^+, X_2^-)$, and so on, with the final result of the five flips being $\ms T'$. If this is the case, then by Proposition~\ref{fliporder}, we have that the order $\le_{34}$ remains unchanged after these flips, and the orders $\le_{12}$, $\le_{13}$, $\le_{14}$, $\le_{32}$, $\le_{42}$ remain unchanged except the order of $f_j$ and $f_{j'}$ is flipped in all of them. By Proposition~\ref{orders}, it follows that $(\ms T')^\text{II} = \emptyset$. This will prove (b).

First, suppose $\ms T$ does not have a flip supported on $(X_1^+,X_1^-)$. By Proposition~\ref{flip}, there is some $\tau \in \ms T(X_1^-)^\ast$ with $\tau \cap X_1^+ = \emptyset$. Let
\[
P = \{ 3 \times f_{j_1}, i_1 \times f_{j_1}, i_1 \times f_{j_2}, i_2 \times f_{j_2}, \dotsc, i_{k-1} \times f_{j_k}, 1 \times f_{j_k} \}
\]
be the path from 3 to 1 in $T(\tau)$. Suppose that $j_k \neq j'$. Since $\tau \cap X_1^+ = \emptyset$, we also have $j_k \neq j$. Since $\{ 3 \times f_j, 1 \times f_{j_k} \} \subseteq \tau$, we have $f_j <_{13} f_{j_k}$. Since $\{ i_{k-1} \times f_{j_k}, 1 \times f_{j'} \} \subseteq \tau$, we have $f_{j_k} <_{1i_{k-1}} f_{j'}$. But $<_{13}$ and $<_{1i_{k-1}}$ are both the same as $<_{12}$, contradicting the assumption that $f_j$ and $f_{j'}$ are consecutive in this order. Hence $j_k = j'$.

Since $\tau \cap X_1^+ = \emptyset$, we have $i_{k-1} \neq 3$. If $i_{k-1} = 4$, then $\{3 \times f_j, 4 \times f_{j'}\} \subseteq \tau$, and thus $f_{j'} <_{34} f_j$. But we assumed that $<_{34}$ is the same as $<_{12}$, so this is a contradiction. Finally, if $i_{k-1} = 2$, then we similarly have $f_{j'} <_{32} f_j$. But $<_{32}$ is the same as $<_{12}$, a contradiction. Thus, $\ms T$ has a flip supported on $(X_1^+,X_1^-)$.

Let the result of this flip be $\ms T_2$. Suppose there is some $\tau \in \ms T(X_2^-)^\ast$ with $\tau \cap X_2^+ = \emptyset$. Let
\[
P = \{ 2 \times f_{j_1}, i_1 \times f_{j_1}, i_1 \times f_{j_2}, i_2 \times f_{j_2}, \dotsc, i_{k-1} \times f_{j_k}, 4 \times f_{j_k} \}
\]
be the path from 2 to 4 in $T(\tau)$. Since the only order $\le_{ii'}$ that changed from $\ms T$ to $\ms T_2$ was $\le_{13}$ (and $\le_{31}$), we can apply analogous arguments to the ones above to show that $j_1 = j$ and $i_1 \neq 4$, 3, 1. This is a contradiction, so $\ms T_2$ has a flip supported on $(X_1^+,X_1^-)$.

Let the result of this flip be $\ms T_3$. Suppose there is some $\tau \in \ms T(X_3^-)^\ast$ with $\tau \cap X_3^+ = \emptyset$. Let
\[
P = \{ 2 \times f_{j_1}, i_1 \times f_{j_1}, i_1 \times f_{j_2}, i_2 \times f_{j_2}, \dotsc, i_{k-1} \times f_{j_k}, 1 \times f_{j_k} \}
\]
be the path from 2 to 1 in $T(\tau)$. Since the only changes from $\ms T$ to $\ms T_3$ in the orders $\le_{ii'}$ were between $f_j$ and $f_{j'}$, we can apply the same arguments as above to show that $j_1 = j$ and $j_k = j'$. Now, since $\tau \cap X_3^+ = \emptyset$, we have $i_1 \neq 1$. Also, since $f_{j'} <_{\ms T_3[13]} f_j$, by Proposition~\ref{compare} we have $i_1 \neq 3$. Hence $i_1 = 4$. Similarly, we have $i_{k-1} = 3$. Thus, $\{ 2 \times f_j, 3 \times f_{j_2} \} \subseteq \tau$, so $f_j <_{32} f_{j_2}$, and $\{ 4 \times f_{j_2}, 1 \times f_{j'} \} \subseteq \tau$, so $f_{j_2} <_{14} f_{j'}$. But both of these orders are the same as $<_{12}$ in $\ms T_3$, which contradicts the fact that $f_j$, $f_{j'}$ are consecutive in this order. Hence $\ms T_3$ has a flip supported on $(X_3^+,X_3^-)$.

Let the result of this flip be $\ms T_4$. Suppose there is some $\tau \in \ms T(X_4^-)^\ast$ with $\tau \cap X_4^+ = \emptyset$. Let
\[
P = \{ 2 \times f_{j_1}, i_1 \times f_{j_1}, i_1 \times f_{j_2}, i_2 \times f_{j_2}, \dotsc, i_{k-1} \times f_{j_k}, 3 \times f_{j_k} \}
\]
be the path from 2 to 3 in $T(\tau)$. First, suppose that $i_r = 1$ for some $r = 1$, \dots, $k-1$. By the argument from part (a) of this proof, we have $f_{j_r} <_{1i_{r-1}} f_{j_{r+1}}$ and $f_{j_{r+1}} <_{1i_{r+1}} f_{j_r}$. If $\{j_r,j_{r+1}\} \neq \{j,j'\}$, then we have a contradiction because $\le_{\ms T_4[1i_{r-1}]}$ and $\le_{\ms T_4[1i_{r+1}]}$ are the same on pairs other than $\{f_j,f_{j'}\}$. Suppose $\{j_r,j_{r+1}\} = \{j,j'\}$. Then we must have $j_r = j_1 = j$ and $j_{r+1} = j_k = j'$. Then the first inequality is $f_j <_{12} f_{j'}$, which contradicts $f_{j'} <_{\ms T_4[12]} f_j$. Hence $i_r \neq 1$ for all $r$.

Now, suppose that $j_k \neq j'$. By the argument we made for $(X_1^+,X_1^-)$, we have $j_k \neq j$, $f_j <_{32} f_{j_k}$, and $f_{j_k} <_{3i_{k-1}} f_{j'}$. Since $i_{k-1} \neq 1$, and we originally assumed that $\le_{34}$ was the same as $\le_{12}$, we have that $<_{3i_{k-1}}$ is the same as $<_{32}$ on all pairs other than $\{f_j,f_{j'}\}$. Hence $f_j <_{32} f_{j_k} <_{32} f_{j'}$, which contradicts the fact that $f_j$ and $f_{j'}$ are consecutive in this order. Thus $j_k = j$. Now, since $\tau \cap X_4^+ = \emptyset$, we have $i_{k-1} \neq 2$. Also, $i_{k-1} \neq 1$ as above. Hence $i_{k-1} = 4$. By Proposition~\ref{compare}, we thus have $f_j <_{42} f_{j'}$. This contradicts the fact that $f_{j'} <_{\ms T_4[42]} f_j$. Hence $\ms T_4$ has a flip supported on $(X_4^+,X_4^-)$.

An analogous argument shows that the result has a flip supported on $(X_5^+,X_5^-)$. This completes the proof of the Theorem.

\end{document}